\documentclass[10pt, a4paper]{amsart}

\usepackage{amsfonts}
\usepackage{graphicx}
\usepackage{epstopdf}
\usepackage{algpseudocode}
\usepackage[T1]{fontenc}
\usepackage{lmodern}
\usepackage[english]{babel}
\usepackage{amssymb}
\usepackage{mathrsfs}
\usepackage{xcolor}
\usepackage[all]{xy}
\usepackage{subcaption}
\usepackage{booktabs}
\usepackage{multirow}
\usepackage{float}
\usepackage{algorithm}
\usepackage{algpseudocode} 
\usepackage{enumitem}
\setlist[description]{leftmargin=2em}
\usepackage{overpic}
\usepackage{url}
\usepackage{hyperref}
\hypersetup{
    colorlinks,
    linkcolor={red!50!black},
    citecolor={blue!50!black},
    urlcolor={blue!80!black}
}
\usepackage{cleveref}

\algrenewcommand\algorithmicwhile{\textbf{While}}
\algrenewcommand\algorithmicfor{\textbf{For}}
\algrenewcommand\algorithmicdo{\textbf{Do}}
\algrenewcommand\algorithmicif{\textbf{If}}
\algrenewcommand\algorithmicthen{\textbf{Then}}
\algrenewcommand\algorithmicelse{\textbf{Else}}
\algrenewcommand\algorithmicend{\textbf{End}}
\algrenewcommand\algorithmicreturn{\textbf{Return}}

\theoremstyle{plain}
\newtheorem{lemma}{Lemma}[section]
\newtheorem{proposition}[lemma]{\textbf{Proposition}}
\newtheorem{theorem}[lemma]{\textbf{Theorem}}

\theoremstyle{definition}
\newtheorem{definition}[lemma]{\textbf{Definition}}
\newtheorem{example}[lemma]{\textbf{Example}}

\newtheorem{remark}[lemma]{Remark}

\newcommand{\N}{\mathbb{N}}

\newcommand{\Q}{\mathbb{Q}}

\newcommand{\C}{\mathbb{C}}

\newcommand{\A}{\mathbb{A}}
\newcommand{\p}{\mathbb{P}}

\newcommand{\tth}{\thinspace}

\newcommand{\mscr}{\mathscr}

\newcommand{\h}{\operatorname{h}}

\newcommand{\pp}[1]{[\![ {#1} ]\!]}

\newcommand{\OO}{\mathscr{O}}
\newcommand{\Hom}{\operatorname{Hom}}
\newcommand{\HHom}{\mathscr{H}\!\!om}
\newcommand{\HH}{\operatorname{H}}

\renewcommand{\tilde}{\widetilde}
\renewcommand{\hat}{\widehat}

\title{Reconstruction of surfaces with ordinary singularities from their silhouettes${}^{\ast}$}
\thanks{${}^{\ast}$ The final version of this work was published on SIAM J.\ Appl.\ Algebra Geometry, 3(3), 472-506, \href{https://doi.org/10.1137/18M1220911}{DOI:10.1137/18M1220911}. This work was supported by the Austrian Science Fund (FWF) through grants W1214-N15 (DK9), P26607, P31061, J4253, and by the Czech Ministry of Education, Youth and Sports through grant LO1506.}

\author{Matteo Gallet}
\address[MG]{Scuola Internazionale Superiore di Studi 
Avanzati/International School for Advanced Studies (SISSA/ISAS),
Via Bonomea 265, 34136 Trieste, Italy}
\email{mgallet@sissa.it}
\author{Niels Lubbes}
\address[NL]{Johann Radon Institute for Computational and Applied 
Mathematics (RICAM), Austrian Academy of Sciences}
\email{niels.lubbes@ricam.oeaw.ac.at}
\author{Josef Schicho}
\address[JS]{Research Institute for Symbolic Computation (RISC), Johannes 
Kepler University}
\email{jschicho@risc.jku.at}
\author{Jan Vr\v{s}ek}
\address[JV]{University of West Bohemia, Faculty of Applied Sciences}
\email{vrsekjan@kma.zcu.cz}

\begin{document}

\begin{abstract}
We present algorithms for reconstructing, up to unavoidable projective 
automorphisms, surfaces with ordinary singularities in three dimensional space 
starting from their silhouette, or ``apparent contour'' --- namely the 
branching locus of a projection on the plane --- and the projection of their 
singular locus.
\end{abstract}

\maketitle

\section{Introduction}

In this paper, we provide an algorithm that deals with the following problem: 
given a homogeneous ternary polynomial~$D$ which is the discriminant of a 
polynomial~$F$, reconstruct~$F$ up to unavoidable automorphisms of the 
polynomial ring that preserve the discriminant. We do not tackle this problem 
in its full generality, and to understand better the conditions that we impose 
on the polynomial~$F$, it is useful to rephrase the question in a geometric 
setting. If we let $S$ be the surface in~$\p^3$ defined by~$F$, and we 
consider a linear projection $\p^3 \dashrightarrow \p^2$, we call \emph{contour} 
the locus of points in~$S$ whose tangent space passes though the center of 
projection. The projection of the contour is the \emph{silhouette} of~$S$, and 
it is the zero set of the discriminant of~$F$ in the direction given by the 
linear projection. The previous problem can then be specified as follows: given 
the silhouette of~$S$ under a projection $\p^3 \dashrightarrow \p^2$, we want to 
reconstruct the surface~$S$ and the projection to~$\p^2$. 
Since we can always precompose a projection by an automorphism of~$\p^3$, we 
can only hope to solve the problem modulo these automorphisms.
We restrict to surfaces that have at most \emph{ordinary singularities}, namely 
those singularities that arise on a general projection to~$\p^3$ of a smooth 
surface living in a higher dimensional projective space. Moreover, we suppose 
that all projections $\p^3 \dashrightarrow \p^2$  we consider are have ``good'' 
properties, namely those that would arise by projecting from a general point. 
This implies that the silhouettes we consider have only ``simple'' 
singularities (see \Cref{figure:ring_torus}).

In its geometric version, the problem we investigate comes within the 
field of \emph{algebraic vision}, namely the study, via algebra and geometry, 
of problems from computer vision. This subject has been investigated 
intensively in the last years; see, for example, the 
books~\cite{Faugeras2001, Hartley2004} and, among others, the 
papers~\cite{Kileel2017, Kohn2017, Joswig2016, Ponce2014, Ponce2017, 
Trager2016}. In particular, the problem of 
reconstructing 3D shapes from 2D information has been investigated thoroughly 
(see for example~\cite{Boyer1997}, \cite{Kang2001} and \cite{Kahl1998}). 
Attempting to reconstruct a 3D surface from just one 2D picture (namely, from 
its silhouette) seems hopeless because small bumps or perturbations in the 
direction of the camera do not leave any trace on the contour. There are 
situations where this approach was tried, but only when strong a priori 
knowledge about the object to reconstruct is available (see~\cite{Ponce1989} 
and~\cite{Zerroug1993}). However, in the algebraic setting, this turns out to be 
doable, due to the rigidity of algebraic varieties. The question of 
reconstruction of a surface from its silhouette was investigated by the 
Italian school of algebraic geometry at the beginning of the twentieth century, 
and culminated with the formulation of Chisini's conjecture and its solution by 
Kulikov in 1999.
\begin{figure}
  \includegraphics[width=\textwidth]{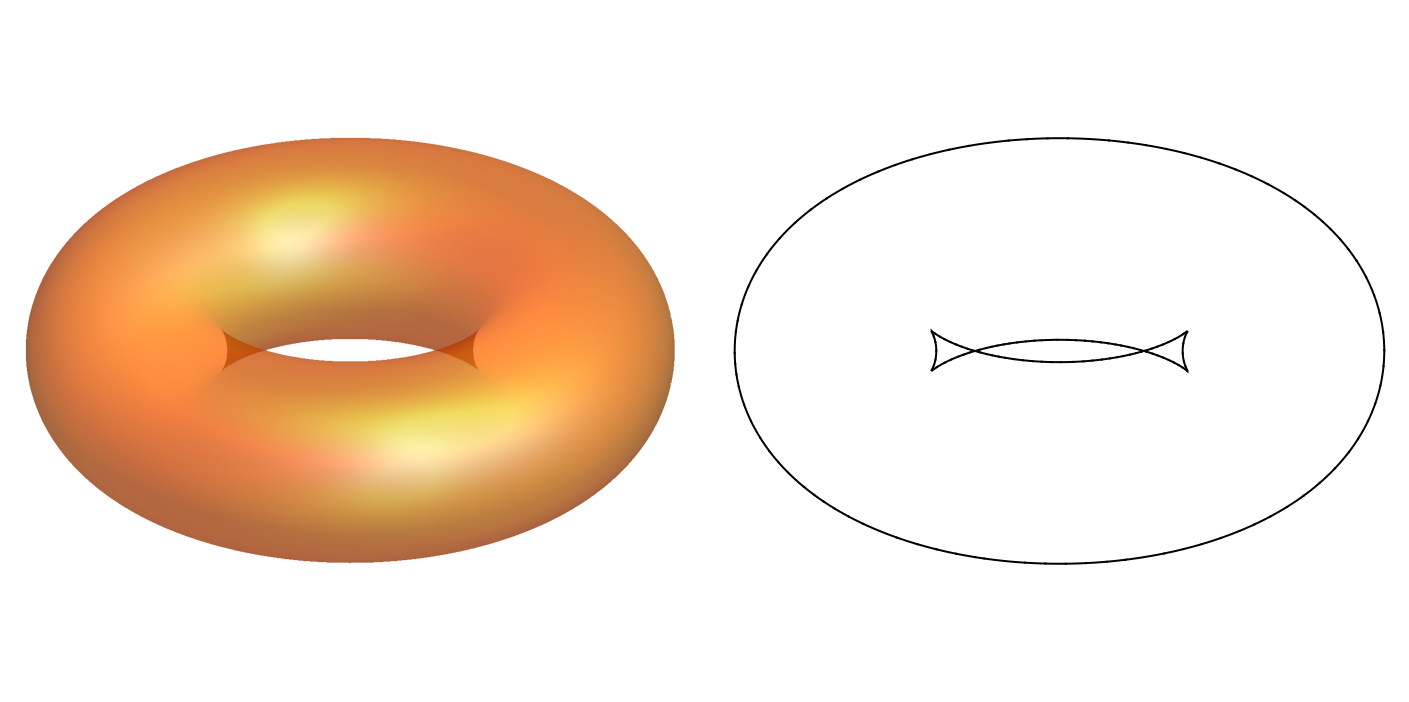}
 \caption{On the left, a ring torus, an algebraic surface of 
degree~$4$. On the right, we highlight its silhouette.}
 \label{figure:ring_torus}
\end{figure}

\subsection*{Chisini's conjecture}

Following the works of Enriques (see~\cite{Enriques1924} and related papers by 
Zariski~\cite{Zariski1929} and Segre~\cite{Segre1930}), Chisini asked 
in~\cite{Chisini1944} whether a surface can be reconstructed from its 
silhouette when it is projected to~$\p^2$. In more modern terms, (see 
\cite[Introduction, Definition~1]{Catanese1986}), one defines a \emph{multiple 
plane} to be a pair~$(S,f)$ where $S$ is a compact connected complex surface and 
$f$ is a finite holomorphic map $f \colon S \longrightarrow \p^2$. The 
pair~$(S,f)$ is said to be \emph{general} if the ramification divisor~$R$ of~$f$ 
is smooth and reduced, $f(R) = B$ has only nodes and ordinary cusps as 
singularities, and $f_{|_R} \colon R \longrightarrow B$ has degree~$1$. Chisini 
conjectured that if two general multiple planes~$(S,f)$ and~$(S', f')$, whose 
maps have degree~$\geq 5$, have the same branching locus~$B \subset \p^2$, 
then there exists an isomorphism $\phi \colon S \longrightarrow S'$ such that 
$f' \circ \phi = f$. Several authors investigated this problem (see 
\cite{Catanese1986, Nemirovskiui2001, Manfredini2002, Moishezon1981} and 
\cite[Section~7.4]{Catanese2008}), until Kulikov solved it in affirmative way 
in~\cite{Kulikov1999} and~\cite{Kulikov2008}. Interestingly, the case when $S$ 
is a smooth surface in~$\p^3$ and $f$ is a general linear projection to~$\p^2$ 
is also solved by Forsyth in~\cite{Forsyth1993}.

\subsection*{Cubic surfaces}

Cubic surfaces in~$\p^3$ are a first non-trivial, though still simple enough, 
case of surface reconstruction from the silhouette. This case was studied by 
Zariski~\cite{Zariski1929} and Segre~\cite{Segre1930}, by Chisini and 
Manara~\cite{Chisini1946}, and by Biggiogero~\cite{Biggiogero1947} (she later 
considered also the case of quartic surfaces in~\cite{Biggiogero1947a}); more 
recently, works focusing on the real situation appeared, see for 
example~\cite{Mikhalkin1995} and \cite{Finashin2015}. A cubic form~$F$ can 
always be brought to Tschirnhaus form $F = w^3 + A(x,y,z) w + B(x,y,z)$ via 
automorphisms of~$\p^3$; its discriminant is $\Delta = -(4 A^3 + 27B^2)$. The 
task of reconstructing the surface $\{ F = 0\}$ from its silhouette is 
equivalent to reconstructing~$A$ and~$B$ from~$\Delta$. Generically, the 
curve~$\Delta$ has six cusps; there is a unique conic passing through those six 
points, which one proves must be~$A$ (possibly up to some scalar multiple). Once 
$A$ is known, the cubic~$B$ can be computed as follows: one selects a cubic~$C$ 
in the ideal of the six points which is linearly independent from the three 
linear multiples of~$A$; then, one makes an ansatz for~$B$ of the form $\lambda 
C + L \cdot A$ where $\lambda \in \C$ and $L$ a linear polynomial, and imposes 
that $-\bigl(4 A^3 + 27 (\lambda C + L \cdot A) \bigr)$ equals the given 
discriminant.

Unfortunately, already for quartic surfaces the formula for the discriminant 
is more complicated, and does not allow a straightforward generalization of the 
procedure for cubics. Nevertheless, the algorithm described in \Cref{reconstruction_smooth} (and already known in the literature) 
provides a generalization of the one for cubics when we restrict to smooth 
surfaces. Going further, the algorithm we present in \Cref{reconstruction_general} applies to even more general situations.

\subsection*{Our contribution}

In this paper, we provide a reconstruction algorithm for surfaces in~$\p^3$ that 
have at most ordinary singularities, namely those singularities that inevitably 
arise when we project a smooth surface in~$\p^5$ to~$\p^3$. 
\Cref{preliminaries} discusses general projections of surfaces with 
ordinary singularities, and in particular describes the possible singularities 
of the silhouette of such projections recalling some well-known classical 
results.

As a warm-up, in \Cref{reconstruction_smooth} we recall the procedure for 
recovering a smooth surface from its silhouette (see~\cite{dAlmeida1992}). We 
proceed in two steps: first, we reconstruct the contour from the silhouette, and 
then we determine the surface. The construction of the contour is based on the 
fact that there is exactly one form~$G_1$ of degree $d^2 - 3d + 2$ vanishing at 
the singularities of the silhouette (in analogy with the existence of the 
conic~$A$ in the situation of cubic surfaces); moreover, there is exactly one 
form~$G_2$ of degree $d^2 - 3d + 3$ vanishing at the singularities of the 
silhouette that is independent from~$G_1$. The contour is the image of the 
silhouette under the rational map
\[
 (x:y:z) \; \mapsto \; (x : y: z : G_2/G_1) \, . 
\]
Once the contour is known, the equation~$F$ of the surface is 
determined so that $F$ and $\partial_w F$ generate the ideal of the contour 
(supposing that the projection is the one along the $w$-axis).

The algorithm for good projections of surfaces with ordinary singularities 
generalizes the one for smooth surfaces. Also here, we use the 
singularities of the silhouette in order to define a rational map that
determines the contour as the image of the silhouette; after that, the 
reconstruction of the surface proceeds exactly as in the smooth situation. One 
important difference with the smooth case is that when dealing with surfaces 
with ordinary singularities we have to take into account the non-reduced
structure of both the contour and the silhouette. Sheaf theory provides a firm 
theoretical ground to prove the correctness of our algorithm, which relies on a 
well-known formula relating the dualizing sheaves of the contour and of the 
silhouette. 

While in the smooth case it is well-known that reconstruction is essentially 
unique, our algorithm could (and in some example really does) give finitely many 
essentially different results. The reason is that the two components of the 
silhouette, which are the projections of the curve of smooth points of the 
surface that are in the contour, and of the singular curve of the surface, may 
intersect transversally; these intersections may either be projections of pinch 
points, or just the result of two distinct points on the surface being collapsed 
by the projection, and it is not possible to distinguish these two situations by 
using only local analytic equations. This ambiguity is related to the failure of 
Chisini's conjecture in low degree (see \cite{Catanese1986}). 
\Cref{figure:ambiguity} shows three essentially different Roman surfaces 
with the same silhouette; in this case, there are projective isomorphisms 
between the surfaces, but none of them preserves the center of projection. This 
case is discussed in more detail in \Cref{example:veronese}.

\begin{figure}
 \begin{center}
  \includegraphics[width=.25\textwidth]{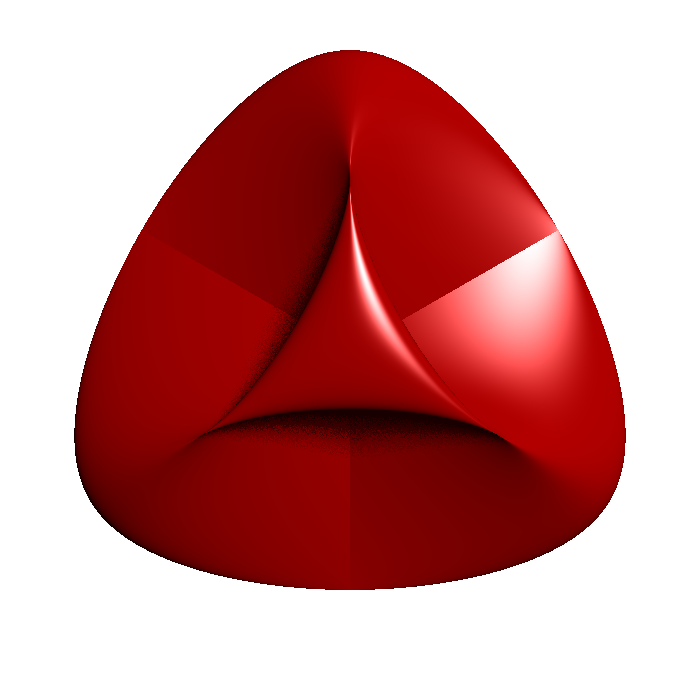} 
  \includegraphics[width=.25\textwidth]{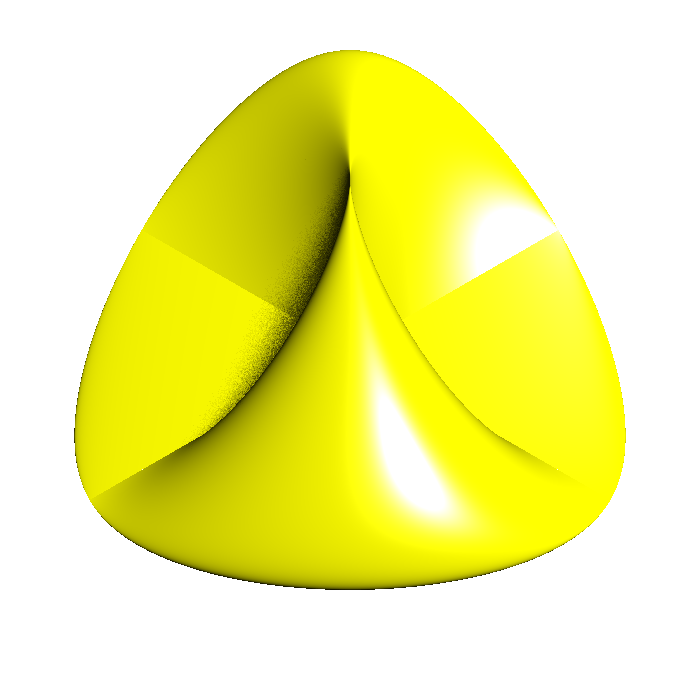} 
  \includegraphics[width=.25\textwidth]{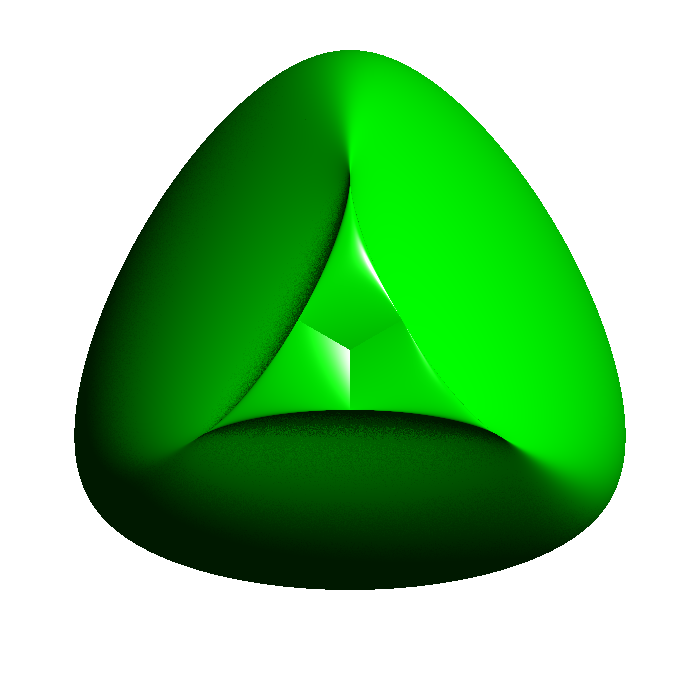} 
 \end{center}
 \caption{Three Roman surfaces with the same silhouette.}
 \label{figure:ambiguity}
\end{figure}

\subsection*{Concerning the algorithm}

An implementation in Maple of our algorithm is available at 
\begin{center}
\url{https://www.risc.jku.at/people/jschicho/pub/Chisini.mpl}. 
\end{center}
The algorithm can easily be re-implemented in any computer algebra system that 
provides Gr\"obner bases. We tested the program for randomly
generated surfaces with different type of singularities of degree up to~$6$; 
the performances are reported at the end of \Cref{reconstruction_general}.

We tried to state the algorithm with as few references to the theory we used to 
prove its correctness as possible, in order to make it available to a wide 
range of readers. The proof of its correctness, instead, requires a basic 
knowledge of sheaf and scheme theory.

The package contains also symbolic proofs that are needed in \Cref{conductor}. 

\section{Singularities of surfaces and their contours and silhouettes}
\label{preliminaries}

In this section we describe the kind of surfaces and projections we are going 
to deal with for the rest of the paper. We fix the following 
terminology. The \emph{contour} of a surface~$S \subset \p^3$ is the 
common zero set of the equation of the surface and its derivative in the 
direction of the projection. The contour is then the union of the singular 
locus~$Z$ of~$S$ and the \emph{proper contour}~$R \subset S$, namely the curve 
of smooth points of the surface whose tangent planes pass through the center of 
projection (see \cite[Remark~3.3]{Ciliberto2011}). 

The \emph{silhouette} is the projection of the contour, and hence it is the 
union of the \emph{singular image}~$W \subset \p^2$, the projection of the 
singular locus, and of the \emph{proper silhouette}~$B \subset \p^2$, the 
projection of the proper contour. If the surface is smooth, the proper contour 
and the proper silhouette are, respectively, what in algebraic geometry are
called the \emph{ramification locus} and the \emph{branching locus} of the 
projection $S \longrightarrow \p^2$ (see \cite[Section~3.1]{Ciliberto2011}).

In our work, we consider surfaces~$S \subset \p^3$ with ordinary singularities 
(see \cite[Definition~7]{Mezzetti1997} and 
\cite[Section~2.1]{Ciliberto2011}). Surfaces with \emph{ordinary singularities} 
are surfaces whose only singularities are self-intersection curves (double 
curves), self-intersection triple points and pinch points (see 
\Cref{figure:general_singularities}). Notice that, in particular, these 
surfaces cannot be tangent developables (this will be useful in the proof of 
\Cref{proposition:projection_good}). Moreover, our object of 
investigation will be  good projections $S \longrightarrow \p^2$. A \emph{good 
projection} is a linear map $S \longrightarrow \p^2$ where $S$ has ordinary 
singularities and such that:
\begin{enumerate}
 \item the restriction of the projection to the contour is injective, except 
for at most finitely many points; \label{enum:injective}
 \item the proper contour~$R$ is smooth, and the proper silhouette~$B$ has at 
most nodes and ordinary cusps; \label{enum:ciliberto}
 \item the line through the center of projection and a point in the
proper contour~$R$ intersects~$S$ with multiplicity exactly~$2$ at that point, 
except for preimages of cusps and singular points on the surface; 
\label{enum:contour}
 \item the singular image~$W$ has only nodes and ordinary triple points ($D_4$ 
singularities), the latter arising as images of spatial triple points; 
\label{enum:triple}
 \item the singular image~$W$ and the proper silhouette~$B$ meet either 
transversally, or tangentially with order~$2$ at smooth points; in particular, 
we ask pinch points to be mapped to transversal intersections.
\label{enum:intersection}
\end{enumerate}

\begin{figure}
  \centering
  \begin{tabular}{ccc}
   \includegraphics[width=.27\textwidth]{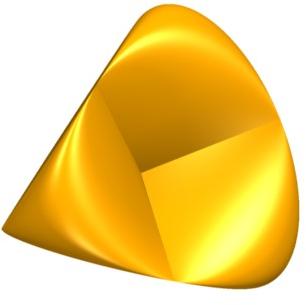} & &
   \includegraphics[width=.3\textwidth]{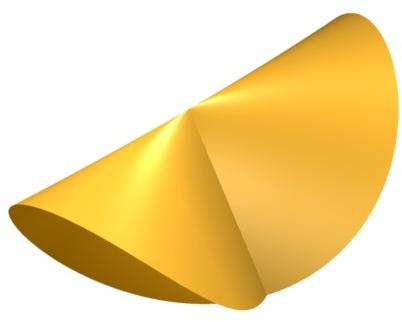}
  \end{tabular}
 \caption{Examples of general singularities of a surface: on the left a 
self-intersection triple point and three pinch points (in a Roman surface) and 
on the right a pinch point (in a Whitney umbrella).}
 \label{figure:general_singularities}
\end{figure}

We remark that the assumptions on the singularities of the surfaces are 
satisfied if the surfaces are general projections of smooth surfaces (see 
\cite[Theorem~8]{Mezzetti1997}). 

We show now that the properties of good projections $S\longrightarrow \p^2$ are 
satisfied if we project a surface with ordinary singularities from a general point 
in~$\p^3$. This is a mild generalization 
of~\cite[Theorem~1.2]{Ciliberto2011}, and in several parts of the proof we use 
the same techniques used by Ciliberto and Flamini. In the proof, we use some 
auxiliary results (\Cref{lemma:secant,lemma:bitangent,lemma:asymptotic,lemma:transverse}) 
which are proved at the end of the section to increase readability since the proof of 
\Cref{proposition:projection_good} is rather long. 

\begin{proposition}
\label{proposition:projection_good}
 If $S \subset \p^3$ is a surface with ordinary singularities, then the 
projection $S \longrightarrow \p^2$ from a general point~$p \in \p^3$ is good.
\end{proposition}
\begin{proof}
 By \cite[Theorem~1.2]{Ciliberto2011}, properties \eqref{enum:injective}, 
\eqref{enum:ciliberto}, and~\eqref{enum:contour} hold for projections from a 
general center.
By assumption, the singular curve~$Z \subset S$ has no singularities other 
than triple points; a general projection may introduce at most nodes and 
project the spatial triple points to planar ordinary triple points. 
Hence condition~\eqref{enum:triple} is satisfied.

In order to ensure condition~\eqref{enum:intersection}, we start by showing 
that no singular point of the singular image~$W \subset \p^2$ lies on the 
proper silhouette~$B \subset \p^2$, and vice versa.

No triple point of the singular image~$W$ lies on~$B$. If the center of 
projection does not lie on any of the three tangent planes at any triple point 
of $Z$, then the contour does not pass through any triple point of~$Z$. If we 
now consider the projection of~$S$ from a triple point, this map has itself a 
silhouette curve, and the cone over this silhouette curve is constituted of 
lines that pass through the triple point and are tangent to~$S$. We get finitely 
many such cones considering all triple points, and if the projection center is 
chosen outside the union of all of them, then no triple point of~$W$ lies on the 
proper silhouette~$B$.

If a node of the singular image~$W$ lies in~$B$, then the projection center 
must lie on a two-secant line of~$Z$ which is tangent to the surface~$S$ at a 
smooth point. \Cref{lemma:secant} states that this does not happen for 
general projection centers, since these lines do not fill~$\p^3$. Recall that, 
having ordinary singularities, the surface~$S$ cannot be a tangent developable, 
and so we can use \Cref{lemma:secant} and the subsequent result. By the 
way, a two-secant line of~$Z$ that is tangent at a singular point of~$S$ would 
be a trisecant of~$Z$, hence lead to a triple point of~$W$ not arising from a 
triple point of~$Z$. This would contradict the first paragraph of the proof.

If a node of~$B$ lies on~$W$, then the projection center lies on a 
bitangent of~$S$ that intersects~$Z$. \Cref{lemma:bitangent} 
excludes that happens for general projection centers.

If a cusp of~$B$ lies on~$W$, then the projection center must lie on an 
asymptotic tangent line (see \cite[Section~3.2]{Ciliberto2011}) that
intersects~$Z$. \Cref{lemma:asymptotic} states that this does not happen 
for general projection centers.

We have established that $B$ and $W$ intersect only at points 
that are smooth in both curves. These intersections arise in two ways: 
projections of an intersection point of~$Z$ and~$R$, or projections of two 
distinct points, one smooth in~$Z$ and another smooth in~$R$. We show that $B$ 
and $W$ intersect transversally in both cases.

Suppose that the intersection is a projection of two distinct points. If the 
intersection were not transversal, then the center of projection would lie on a 
line~$L$ that  
intersects~$Z$ at a smooth point~$q$ and such that the tangent to~$Z$ at~$q$ is 
contained in a tangent plane of~$S$ at a smooth point contained in~$L$.
This is excluded by \Cref{lemma:transverse}.

Suppose that the intersection of~$B$ and~$W$ comes from an intersection of~$R$ 
and~$Z$. The strategy here to show that $B$ and $W$ intersect transversally is 
the following: we first show that the intersection between~$R$ and~$Z$ must be 
transversal; then $R \cap Z$ can be either a simple double point of~$R \cup Z$, 
or a pinch point. In the first case, we show that the point is mapped to a point 
of simple tangency; in the second case, we prove that the transversality of the 
intersection is preserved by the projection.

We begin with the first step, namely showing that the intersection between~$R$ 
and~$Z$ is transversal. We start by analyzing the tangent directions of 
the proper contour. Suppose that the surface~$S$ is defined by a polynomial $F 
\in \C[x,y,z,w]$. Then the contour is defined by~$F$ and by the polynomial $a 
F_x + b F_y + c F_z + d F_w$ where $(a:b:c:d)$ is the center of projection. If 
we de-homogenize setting $w=d=1$, then the equations for the contour are
\[
 F = 0, \quad
 (x-a)F_x + (y-b)F_y + (z-c) F_z = 0\,. 
\]
The tangent direction of the contour at $(x,y,z)$ is then given by the vector 
product of the gradients
of the two equations:
\begin{multline*}
 \begin{pmatrix}
  F_x \\ F_y \\ F_z
 \end{pmatrix}
 \times
 \begin{pmatrix}
  F_x  + (x-a) F_{xx} + (y-b) F_{xy} + (z-c) F_{xz} \\
  F_y  + (x-a) F_{yx} + (y-b) F_{yy} + (z-c) F_{yz} \\
  F_z  + (x-a) F_{zx} + (y-b) F_{zy} + (z-c) F_{zz}
 \end{pmatrix} = \\
 = \nabla(F) \times \Bigl( H(F) \cdot \bigl( (x,y,z) - (a,b,c) \bigr) \Bigr) \,.
\end{multline*}
Let $P$ be an intersection point of $Z$ and $R$. Notice that the surface~$S$ 
has two (analytic) branches around~$P$. From now on, we focus the branch of 
the surface~$S$ at~$P$ containing the proper contour~$R$, and we want to 
understand when the proper contour~$R$ of~$S$ is tangent to the singular 
locus~$Z$. Hence, from now on we let $F$ denote the analytic equation of the 
branch of~$S$ at~$P$ that contains~$R$. By a linear 
change of coordinates, we can assume that $P = (0,0,0)$ and $\nabla(F)|_{P} = 
(0,0,1)$, and that the tangent line~$T_P Z$ is spanned by~$(1,0,0)$. Locally 
at~$P$, the affine equation of the branch of~$S$ containing the proper contour 
is of the form $F(x,y,z) = z - f(x,y)$; moreover the center of the projection 
has coordinates $(a, b, 0)$, otherwise the proper contour would not pass 
through~$P$. Then
\[
 H(F) =
 \begin{pmatrix}
  f_{xx} & f_{xy} & 0 \\
  f_{xy} & f_{yy} & 0 \\
  0 & 0 & 0
 \end{pmatrix} \, .
\]
The direction of the contour at $P$ is then
\[
 \begin{pmatrix}
  0 \\ 0 \\ 1
 \end{pmatrix}
 \times
 \begin{pmatrix}
  f_{xx} \cdot a + f_{xy} \cdot b \\
  f_{xy} \cdot a + f_{yy} \cdot b \\
  0
 \end{pmatrix} \,.
\]
Hence the proper contour is tangent to the singular locus at~$P$ if and only if 
\[
 a \, f_{xx} + b f_{xy} = 0 \, ,
\]
since we supposed that the tangent direction of~$Z$ at~$P$ is $(1,0,0)$. We 
distinguish three situations:
\begin{description}
 \item[$\mathrm{rk}\bigl(H(f)\bigr) = 0$] in this case the tangency condition 
is satisfied for every~$(a,b)$.
 \item[$\mathrm{rk}\bigl(H(f)\bigr) = 1$] in this case we have a so-called 
\emph{parabolic point}; the Hessian of~$f$ is of the form 
$\left( \begin{smallmatrix} \alpha^2 & \alpha \beta \\ \alpha \beta & \beta^2 
\end{smallmatrix} \right)$, and so it has a one-dimensional kernel, also called
the asymptotic direction of the parabolic point.
If the tangent direction of~$Z$ lies in this kernel, then the tangency 
condition is satisfied for every~$(a,b)$.
 \item[$\mathrm{rk}\bigl(H(f)\bigr) = 2$] in this case the tangency condition 
is not satisfied for a general choice of $(a,b)$.
\end{description}
We consider points with zero Hessian as degenerate parabolic points with 
infinitely many asymptotic directions, to avoid a case distinction in the rest 
of the proof. We claim that a curve of parabolic points, whose tangent 
direction is always the/a asymptotic direction, has the property that the 
tangent plane is constant along the curve. Recall that, locally, the surface has 
equation $z - f(x,y) = 0$. We can locally define the curve by an additional 
second equation $y - h(x) = 0$. We want to show that the gradient vector
\[  
 \begin{pmatrix}
  f_{x} \bigl( x, h(x) \bigr) \\
  f_{y} \bigl( x, h(x) \bigr) \\
  -1
 \end{pmatrix} 
\]
is constant. The derivative of this expression is
\[ 
 \begin{pmatrix}
  f_{xx} \bigl( x,h(x) \bigr) + f_{xy} \bigl( x, h(x) \bigr) h'(x) \\
  f_{xy} \bigl( x, h(x) \bigr) + f_{yy} \bigl( x, h(x) \bigr) h'(x) \\
  0
 \end{pmatrix} = 
 \begin{pmatrix}
  H(f)\bigl( x, h(x) \bigr) 
   \begin{pmatrix} 1 \\ h'(x) \end{pmatrix} \\
  0
 \end{pmatrix} \, , 
\]
which is zero by assumption. The claim, namely the fact that the tangent plane 
is constant along the curve, is thus proven. Notice that, since as we 
already remarked a surface with ordinary singularities cannot be a 
tangent developable, not all points on the surface are parabolic. Hence, if the 
center of projection is outside these finitely many planes determined by curves 
of parabolic points or isolated parabolic points, the curves~$R$ and~$Z$ will 
intersect transversally.

We now show that if a point of intersection of~$R$ and~$Z$ is a pinch point, 
then its projection is a point of transverse intersection between~$B$ and~$W$; 
moreover, we show that if an intersection of~$R$ and~$Z$ is not a pinch point, 
then its projection is a point of simple tangential intersection of~$B$ and~$W$. 
Once we prove this, condition~\eqref{enum:intersection} is ensured and the whole 
proof is concluded.

Suppose that $P \in R \cap Z$ is not a pinch point. Locally around~$P$, we can 
take analytic coordinates such that the proper contour is defined by $x = z = 
0$, the branch of~$S$ containing it has equation $x - z^2 = 0$, and 
the projection is along the $z$-axis. Since $R$ and $Z$ intersect transversally, 
there exists a power series~$h$ of positive order such that the equation of the 
singular locus~$Z$ is of the form $y - h(z) = x - z^2 = 0$. The equation of the 
proper silhouette is $x = 0$. The equation of the singular image is given by 
eliminating~$z$ from the equations $y - h(z) = 0$ and $x - z^2= 0$. We 
write~$h(z)$ in the form $z h_1(z^2) + z^2 h_2(z^2)$, so the 
elimination ideal is generated by 
\[ 
 \bigl( y + x h_2(x) \bigr)^2 - x h_1(x)^2 \, .  
\]
This shows that $B$ and $W$ have the same linear factor, so they are tangent, 
but if we set $x = 0$ in the equation of~$W$ we obtain a non-zero quadratic 
summand, proving that the tangency is simple.

Consider now the case that $P \in R \cap Z$ is a pinch point. Recall that a 
pinch point is a singular point such that the analytic germ of the surface at 
the point has equation equivalent to $z^2 - y x^2 = 0$, see 
\cite[Definition 2.1]{Ciliberto2011}. We know 
that pinch points are double points. Hence, for a general projection for which 
we choose coordinates  $(x,y,z) \mapsto (x,y)$, we have that $S$ has local 
equation at~$P$ of the form $z^2 + h_1(x,y) z + h_2(x,y) = 0$. A Tschirnhaus 
transformation $z \mapsto z - h_1(x,y) / 2$, which leaves the direction of 
projection invariant, makes the local equation of~$S$ in the form $z^2 + h(x,y) 
= 0$. Now, pinch points can be characterized as points such that the 
discriminant of a general projection is the product of  a square of a linear 
factor and another linear factor intersecting transversally the first one, 
namely it is of the form $u^2 v$. In these coordinates, hence, the surface~$S$ 
has equation $z^2 + u^2 v = 0$ at~$P$, and the projection can still be assumed 
to be along the $z$-axis. The contour is then given by $z^2 + u^2 v = z = 0$, 
so we see that the projection maps it isomorphically to the plane 
curve~$u^2 v = 0$. This concludes the proof that pinch points project to 
transverse intersections of the proper silhouette and the singular image.  
\end{proof}

In order to prove the auxiliary results needed for 
\Cref{proposition:projection_good} we use the results from 
\emph{focal geometry} introduced and proved in \cite[Sections~4 
and~5]{Ciliberto2011}. Here we briefly sketch the setting and the results, and 
we refer to the work of Ciliberto and Flamini for more precise information. We 
consider families of lines in~$\p^3$, namely varieties $\mscr{X} \subset D 
\times \p^3$, where $D \subset \mathbb{G}(1,3)$ is a subvariety of the 
Grassmannian of lines in~$\p^3$, of the form
\[
 \mscr{X} = 
  \bigl\{ 
   (\delta, x) \, \colon \, 
   \delta \in D, \, 
   x \text{ belongs to the line corresponding to } \delta 
  \bigr\} \, .
\]
By restricting the second projection to~$\mscr{X}$, we get a map $f \colon 
\mscr{X} \longrightarrow \p^3$; its ramification points form the \emph{focal 
locus} of~$\mscr{X}$. We say that $\mscr{X}$ is a \emph{filling family} if $D$ 
is two-dimensional and $f$ is dominant. 

\begin{theorem}
\label{theorem:focal}
Let $f \colon \mscr{X} \longrightarrow \p^3$ be a filling family, let $S 
\subset \p^3$ be a non-developable surface, and let $Z \subset \p^3$ be a 
curve. For a general element $\delta \in D$, the fiber 
\[
 \mscr{X}_\delta = 
  \bigl\{ 
   (\delta, x) \in \mscr{X} \, \colon \, 
   x \text{ belongs to the line corresponding to } \delta 
  \bigr\}
\]
intersects the focal locus in two points (or one counted with 
multiplicity~$2$). 

Moreover, if for a general element $\delta \in D$ the line 
$\ell = f(\mscr{X}_\delta)$ intersects the surface~$S$ 
tangentially at a point~$p$, then the following properties hold:
\begin{itemize}
 \item[(a)] the point $(\delta, p) \in \mscr{X}_\delta$ is a focus, namely a 
point in the focal locus;
 \item[(b)] the multiplicity of intersection of~$\ell$ with~$S$ at~$p$ is 
at most~$3$;
 \item[(c)] if the multiplicity of intersection of~$\ell$ with~$S$ at~$p$ 
is~$3$, then $(\delta, p)$ is a focus of~$\mscr{X}_\delta$ of multiplicity~$2$.
\end{itemize}
Moreover, if for a general element $\delta \in D$ the line $\ell = 
f(\mscr{X}_\delta)$ intersects the curve~$Z$ in a point~$p$, then 
$(\delta,p)$ is a focus in~$\mscr{X}_\delta$.
\end{theorem} 

Notice that the last statement in \Cref{theorem:focal} is not 
present in~\cite{Ciliberto2011} but can be proven in an analogous way. 
Moreover, in~\cite{Ciliberto2011} it is used the notion of \emph{contact 
order} instead of \emph{multiplicity of intersection}: however, these two 
numbers just differ by~$1$, so we adopt the latter notion. 

With these results at hand, we can proceed with proving our auxiliary lemmas. 
Here, the considered families of lines~$\mscr{X}$ are families whose general 
element is tangent or bitangent to a given non-developable surface~$S$ and 
intersects a given curve~$Z$.

\begin{lemma} \label{lemma:secant}
Let $S \subset \p^3$ be a non-developable surface and let $Z \subset \p^3$ be a 
curve. Then the family of two-secant lines of~$Z$ that are tangent to the 
surface~$S$ at a smooth point does not fill~$\p^3$.
\end{lemma}
\begin{proof}
 If such a family were filling, then any of its general members would carry 
three foci, which is impossible by \Cref{theorem:focal}.
\end{proof}

\begin{lemma} \label{lemma:bitangent}
Let $S \subset \p^3$ be a non-developable surface and let $Z \subset \p^3$ be a 
curve. Then the family of bitangents of~$S$ that intersect~$Z$ does not 
fill~$\p^3$.
\end{lemma}
\begin{proof}
 If such a family were filling, then any of its general members would carry 
three foci, which is impossible by \Cref{theorem:focal}.
\end{proof}

\begin{lemma} \label{lemma:asymptotic}
Let $S \subset \p^3$ be a non-developable surface and let $Z \subset \p^3$ be a 
curve. Then the family of asymptotic tangent lines of~$S$ that intersect~$Z$ 
does not fill~$\p^3$.
\end{lemma}
\begin{proof}
 If such a family were filling, then any of its general members would carry 
two foci, one of which with multiplicity~$2$, which is impossible by 
\Cref{theorem:focal}.
\end{proof}

\begin{lemma} \label{lemma:transverse}
Let $S \subset \p^3$ be a non-developable surface and let $Z \subset \p^3$ be a 
curve. Then the family of lines~$L$ that intersect~$Z$ at a smooth point~$q$ and 
such that the tangent to~$Z$ at~$q$ is contained in a tangent plane of~$S$ at a 
smooth point contained in~$L$ does not fill~$\p^3$.
\end{lemma}

\begin{proof}
Assume indirectly that the family of lines is filling. Let $\mscr{Y}$ be the 
family of tangent planes to~$S$ whose existence is postulated by the assumption 
(these planes have to be tangent to~$Z$ as well). We distinguish two cases. 
First, suppose that $\mscr{Y}$ is two-dimensional. The family~$\mscr{Y}$ is 
contained in the two-dimensional family of tangent planes to~$S$, and in the 
two-dimensional family of tangent planes to~$Z$, and both families are 
irreducible. Moreover, the second family forms a tangent developable surface in 
the dual projective space, while the first one does not. So the two irreducible 
families cannot be equal and therefore intersect in a family of dimension one, 
which contradicts the assumption. Second, suppose that $\mscr{Y}$ is 
one-dimensional. Then there are infinitely many lines of the filling family 
contained in a general plane in~$\mscr{Y}$. It follows that there are infinitely 
many points at which such a plane is tangent to~$S$. Then the surface~$S$ has 
only a one-dimensional family of tangent planes, which implies that it is a 
developable surface. This contradicts the assumption.
\end{proof}

To sum up, suppose we have a good projection $S \longrightarrow \p^2$. If $B$ 
is the proper silhouette and $W$ is the singular image of the surface~$S$, then 
the curve~$B \cup W$ has only the following $7$ types of singularities, which 
we call \emph{special points} (see \Cref{figure:singularities}):
\begin{itemize}
 \item[-] nodes or cusps of~$B$,
 \item[-] nodes or triple points of~$W$,
 \item[-] tangential intersections of~$B$ and~$W$,
 \item[-] transversal intersections of~$B$ and~$W$ whose preimages are distinct,
 \item[-] transversal intersections of~$B$ and~$W$ coming from pinch points.
\end{itemize}
\begin{figure}
 \centering
 \begin{tabular}{cccc}
  \includegraphics[width=.18\textwidth]{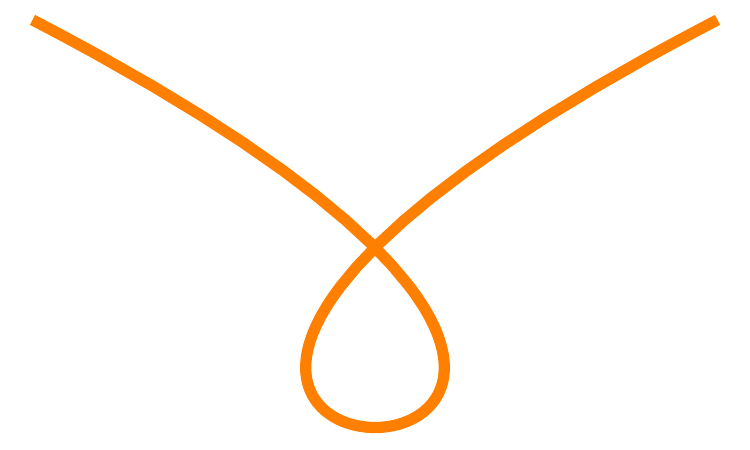} &
  \includegraphics[width=.18\textwidth]{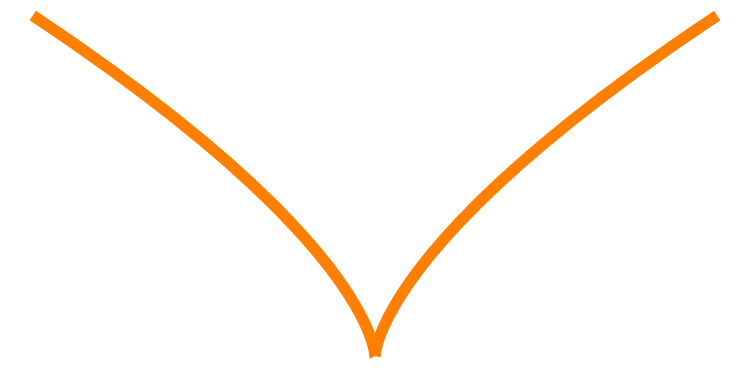} &
  \includegraphics[width=.18\textwidth]{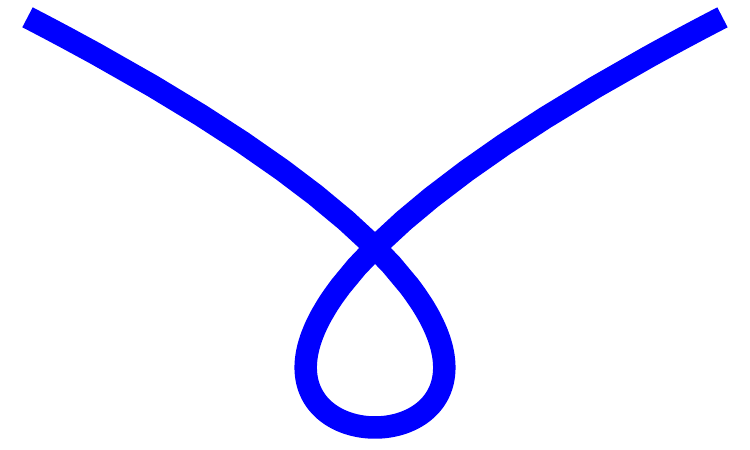} &
  \includegraphics[width=.18\textwidth, height=1.6cm]{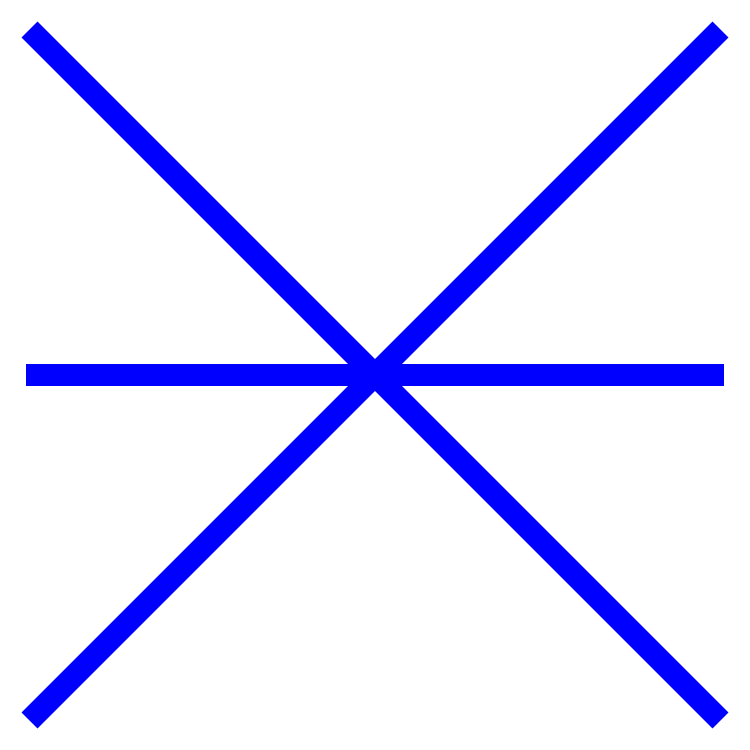}
 \end{tabular}
 
 \begin{tabular}{ccc}
  \includegraphics[width=.18\textwidth, height=2cm]{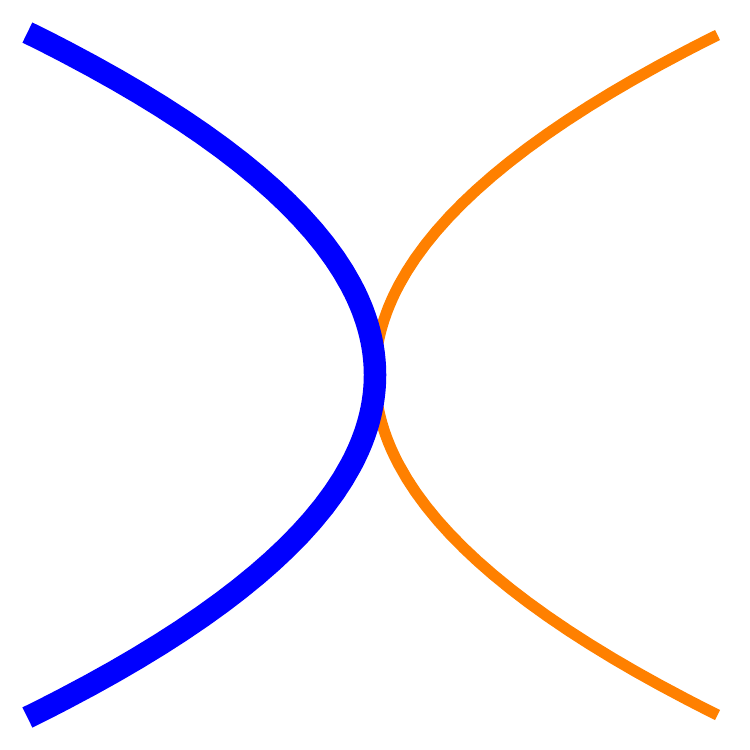} &
  \includegraphics[width=.18\textwidth, height=2cm]{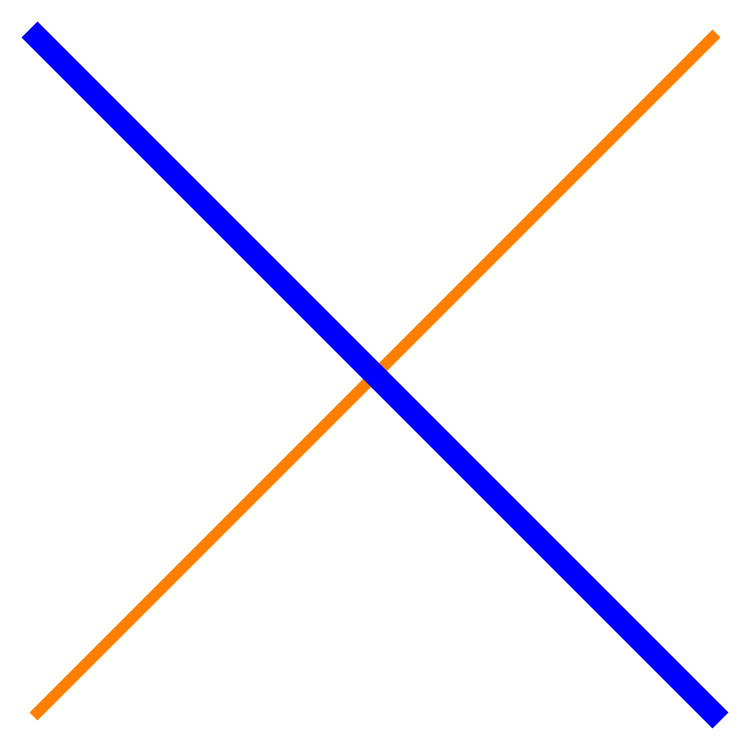} &
  \includegraphics[width=.18\textwidth, height=2cm]{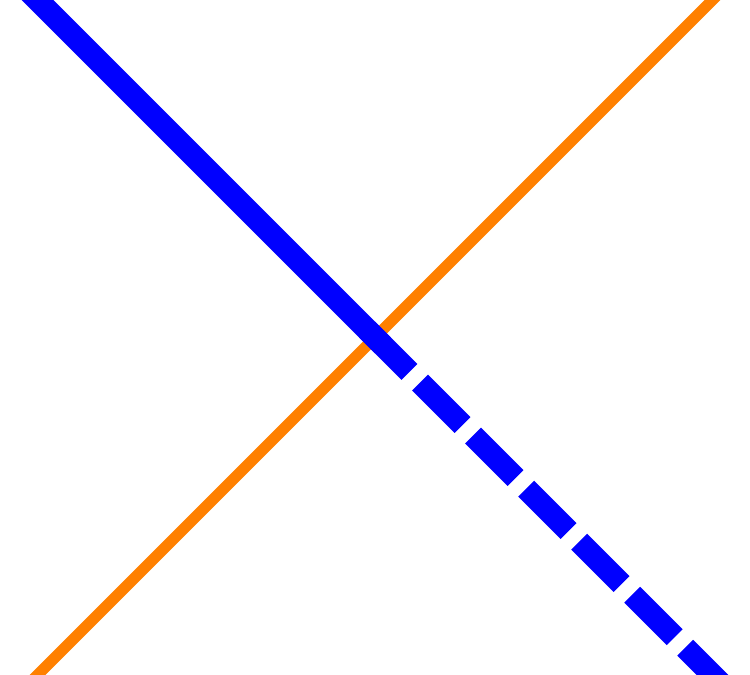}
 \end{tabular}
 \caption{The seven possible singularities of the union of the proper 
silhouette (thinner, in orange) and the singular image (thicker, 
in blue) of a surface in~$\p^3$. The case of a singularity coming from a pinch 
point of the surface is denoted by a dotted line.}
 \label{figure:singularities}
\end{figure}

\section{Reconstruction of smooth surfaces}
\label{reconstruction_smooth}

The question of reconstructing a smooth surface from its silhouette has been 
answered by d'Almeida in~\cite{dAlmeida1992}. We report his construction --- 
without any claim of originality --- because it introduces several key concepts 
that will be used later in \Cref{reconstruction_general} to deal with the 
more general case of surfaces with ordinary singularities.

The silhouette of a good projection of a smooth surface in~$\p^3$ of degree~$d$ 
is a curve of degree~$d(d-1)$ with only nodes and cusps as singularities (see 
\Cref{figure:smooth_quartic}). The contour, also of degree~$d(d-1)$, is a 
smooth curve which is a complete intersection and hence it is \emph{linearly 
normal}, namely it is not the projection of a non-degenerate curve living in a 
bigger projective space. Therefore, we can reconstruct the contour from the 
silhouette as its linear normalization (see \cite[Definition~2.11]{Zak1993}). 
Once we have access to the ideal of the contour, the unique form of degree~$d-1$ 
must be the derivative in the direction of the projection of the yet 
to-be-determined equation of the surface. Finding such an equation becomes then 
a problem in linear algebra, which admits a unique solution.
\begin{figure}
 \centering
 \begin{tabular}{cc}
  \includegraphics[width=.4\textwidth]{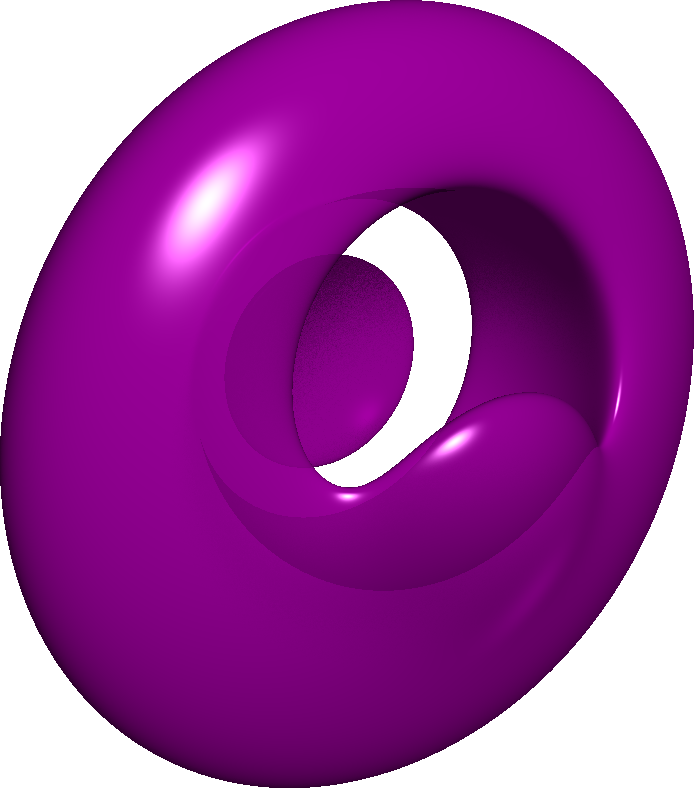}
  &
  \includegraphics[width=.4\textwidth]{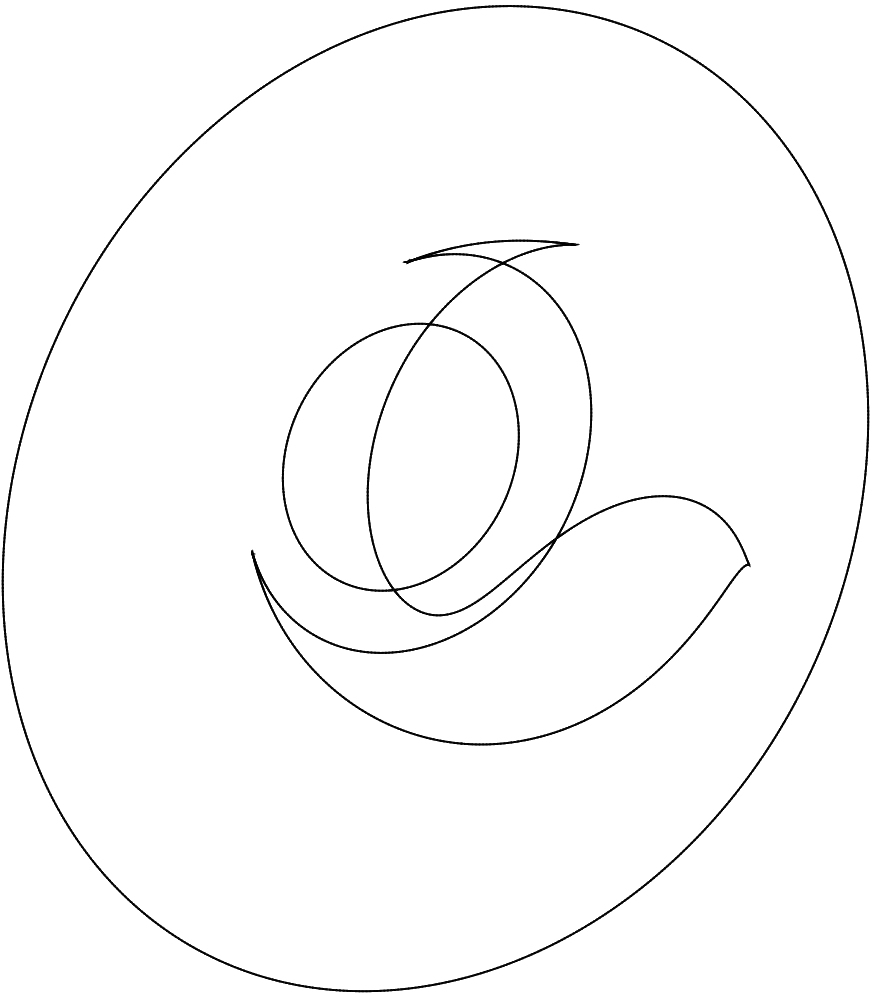}
 \end{tabular}
 \caption{A smooth surface of degree~$4$ (on the left) and its silhouette in 
the plane (on the right).}
 \label{figure:smooth_quartic}
\end{figure}

We start by the reconstruction of the contour. 

\begin{remark}
\label{remark:strategy}
The key fact here is that 
$\OO_{R}(1)$, the line bundle embedding the contour~$R$ in~$\p^3$, is a twist 
of the canonical sheaf~$\omega_R$ of~$R$; by the theory of adjoints, one proves 
that $\pi_{\ast} (\omega_R)$ is a twist of the ideal of singularities of 
the silhouette~$B$.
 The global sections of~$\pi_{\ast}  \OO_{R}(1)$ can then be obtained as  
homogeneous forms of a certain degree passing through the singularities of~$B$. 
In this way we get a way to map~$B$ into~$\p^3$ whose image is projectively 
equivalent to~$R$.
\end{remark}

\begin{lemma}
\label{lemma:linearly_normal}
 The contour~$R$ of a good projection is linearly normal. This means that the 
standard map $\HH^0 \bigl( \p^3, \OO_{\p^3}(1) \bigr) \longrightarrow \HH^0 
\bigl( R, \OO_{R}(1) \bigr)$ is an isomorphism. In particular, $\HH^0 
\bigl( R, \OO_{R}(1) \bigr)$ is $4$-dimensional.
\end{lemma}
\begin{proof}
 Since $R$ is a smooth complete intersection, it is linearly normal.
\end{proof}

\begin{lemma}
\label{lemma:canonical_contour}
The canonical sheaf~$\omega_R$ of~$R$ is isomorphic to~$\OO_{R}(2d - 5)$. 
Moreover, the canonical sheaf~$\omega_{B} = \pi_{\ast}(\omega_R)$ of the 
silhouette~$B$ is isomorphic to~$\mscr{J}\bigl(d^2 - d - 3\bigr)$, where 
$\mscr{J}$ is the restriction to~$B$ of the ideal sheaf~$\mscr{K}$ 
on~$\p^2$ of the singularities of~$B$.
\end{lemma}
\begin{proof}
  The statement regarding the canonical sheaf of~$R$ follows from the fact that 
$R$ is the complete intersection of two surfaces of degree~$d$ and~$d-1$, and 
from the adjunction formula, see \cite[Exercise~II.8.4e]{Hartshorne1977}. Since 
$B$ has degree~$d(d-1)$, the theory of adjoints for plane curves shows that 
\[
 \omega_B \cong \mscr{J}\bigl(d(d-1) - 3\bigr) 
 = 
 \mscr{J}\bigl(d^2 - d - 3\bigr) \, ,
\]
see \cite[Chapter~8, Proposition~8]{Fulton1989} for the case of curves with 
only nodes, the situation of cusps is analogous.
\end{proof}

\begin{proposition}
\label{proposition:embedding_contour}
 The complete linear series~$|\omega_B(-2d + 6) |$ maps~$B$ to~$\p^3$, and 
the image of this map is, up to projective equivalence in~$\p^3$ over~$B$, 
equal to~$R$. These linear series correspond to global sections of 
$\mscr{J}(d^2-3d+3)$.
\end{proposition}
\begin{proof}
 We showed in \Cref{lemma:canonical_contour} that there is an isomorphism 
$\omega_R(-2d+6) \cong \OO_R(1)$. Recall that the latter divisor is the one 
providing the embedding of the contour~$R$ in~$\p^3$, and in this embedding $R$ 
is linearly normal. The projection $R \longrightarrow B$ determines an 
isomorphism between the global sections of~$\omega_R$ and $\omega_B$. By 
construction, the image of~$B$ under the complete linear series~$| \omega_B(-2d 
+ 6) |$ is also linearly normal, and so must coincide up to projective 
equivalence over~$B$ with~$R$. The last statement follows from the second part 
of \Cref{lemma:canonical_contour}.
\end{proof}

Since $\omega_R(-2d+6) \cong \OO_R(1)$, it follows that $\h^0 \bigl( 
\omega_R(-2d+6) \bigr) = 4$. Thus, there are exactly~$4$ 
linearly independent forms of degree $d^2 - 3d + 3$ in the ideal~$J$ defining 
the sheaf~$\mscr{J}$. Since $\omega_R(-2d+5) \cong \OO_R$, it follows that 
$\omega_B(-2d+5) \cong \OO_B$, and so $\mscr{J}(d^2 - 3d + 2)$ has a 
one-dimensional space of global sections. 

Notice that for all $n \in \N$ there is the following exact sequence:
\[
 \xymatrix{
  0 \ar[r] & \mscr{I}(n) \ar[r] & \mscr{K}(n) \ar[r] & \mscr{J}(n)\ar[r] & 0 
 } \, ,
\]
where $\mscr{I}$ is the ideal sheaf of~$B$ on~$\p^2$. Taking global sections, 
we get:
\[
 \xymatrix{
  0 \ar[r] & \operatorname{H}^0 \bigl(\mscr{I}(n) \bigr) \ar[r] &
  \operatorname{H}^0 \bigl(\mscr{K}(n) \bigr) \ar[r] &
  \operatorname{H}^0 \bigl(\mscr{J}(n) \bigr) \ar[r] &
  \operatorname{H}^1 \bigl(\mscr{I}(n) \bigr)
 } \, .
\]
Since $B$ has degree~$d(d-1)$, we have $\mscr{I}(n) \cong 
\OO_{\p^2}(-d(d-1)+n)$. It follows that $\operatorname{H}^1\bigl(\mscr{I}(n) 
\bigr)=0$ for all $n \in \N$. Thus, global sections of~$\mscr{J}(n)$ are 
restrictions of global sections of~$\mscr{K}(n)$ in~$\p^2$.
Hence there exists a unique (up to scalars) form~$G_1$ of degree $d^2 - 3d + 
2$ in the ideal~$K$ of singularities of~$B$, and there is a unique form~$G_2$ 
of degree $d^2 - 3d + 3$ up to scalars and multiples of~$G_1$.

\Cref{proposition:embedding_contour} implies that the contour~$R$ 
can be obtained by mapping the silhouette~$B$ via the rational map from~$\p^2$ 
to~$\p^3$ given by three multiples of~$G_1$ by linearly independent linear 
forms, and~$G_2$, see Steps~$2$ and~$3$ in Algorithm 
\texttt{ReconstructSmoothSurface}. If we take coordinates so that the 
projection 
$S \longrightarrow \p^2$ is the map forgetting the last coordinate, then the 
three linear forms can be taken to be $x$, $y$ and $z$; in this way, the 
rational map $\p^2 \dashrightarrow \p^3$ is
\[ (x:y:z) \mapsto (x:y:z:G_2/G_1) \]
and it is a section of the projection, see Step~$4$ of Algorithm 
\texttt{ReconstructSmoothSurface}.

Once the contour is reconstructed, let $I$ be its homogeneous ideal 
in~$\C[x,y,z,w]$. By hypothesis, the minimal non-zero 
homogeneous component of~$I$ is the one in degree~$d-1$. This component is 
one-dimensional, hence the derivative~$H$ of the equation of the surface in the 
direction of the projection is uniquely determined up to scalars. Now, it is 
enough to compute a form~$F$ of degree~$d$ in~$I$ such that its derivative 
is~$H$. This amounts to solving a system of linear equations, 
see Steps~$5$ and~$6$ of Algorithm \texttt{ReconstructSmoothSurface}. In fact, 
suppose that the projection direction is the one along the $w$-axis; by 
integration we can compute a primitive~$\tilde{H}$ of~$H$; then we make an 
ansatz for the integration constant, which must be a homogeneous polynomial~$N$ 
of degree~$d$ depending only on~$x$, $y$ and $z$. Reducing the polynomial 
$\tilde{H} + N$ modulo a Gr\"obner basis of~$I$ gives linear equations for the 
coefficients of~$N$.

\textbf{Claim.} This linear system has a unique solution.

\textit{Proof.} Suppose that $F_1$ and $F_2$ are two different solutions; then 
there are constants~$a$ and~$b$ such that $F := a F_1 + b F_2$ is an element 
of~$I$ such that its derivative 
along the direction of the projection is zero. 
This means that $F$ is the equation of a cone of degree~$d$ passing through the 
contour~$R$ whose vertex is the projection center. The projection of the 
cone would be a component of degree~$d$ of the silhouette. This is absurd 
because the silhouette is irreducible of degree~$d(d-1)$. \hfill $\square$

This proves that Algorithm \texttt{ReconstructSmoothSurface} is correct and 
that every smooth surface having branching locus~$B$ is projectively 
equivalent over~$B$ to the output.

\begin{algorithm}
\caption{\texttt{ReconstructSmoothSurface}}
\begin{algorithmic}[1]
  \Require A curve $B \subset \p^2$, the silhouette of a good projection 
to~$\p^2$ of a smooth surface~$S \subset \p^3$ of degree~$d$.
  \Ensure A smooth surface~$S \subset \p^3$ together with a projection 
to~$\p^2$ such that $B$ is the branching locus of this projection.
  \Statex
  \State {\bfseries Compute} the radical $K$ of the Jacobian ideal 
of~$B$.
  \State {\bfseries Select} in $K$ a form $G_1$ of degree $d^2-3d+2$.
  \State {\bfseries Select} in $K$ a form $G_2$ of degree $d^2-3d+3$ 
which is not a multiple of~$G_1$.
  \State {\bfseries Compute} the ideal~$I$ of the image $R$ of $B$ under the 
map \[ (x:y:z) \mapsto (x:y:z:G_2/G_1). \]
  \State {\bfseries Select} in $I$ a form~$H$ of degree~$d-1$.
  \State {\bfseries Select} in $I$ a form~$F$ whose derivative is a scalar 
multiple of $H$.
  \State \Return $F$.
\end{algorithmic}
\end{algorithm}

\section{Reconstruction of surfaces with ordinary singularities}
\label{reconstruction_general}

In this section we present a reconstruction algorithm for good projections $S 
\longrightarrow \p^2$. It subsumes the previous case presented in 
\Cref{reconstruction_smooth}. The idea is similar to the one in the 
smooth case: we first reconstruct the contour, and then we obtain the surface 
via linear algebra. However, now it is not enough to compute the normalization 
of the silhouette, because the contour may be singular. Instead, we solve local 
reconstruction problems for each of the seven types of special points that can 
arise in the silhouette and obtain the global result by sheaf theory.

Recall that we denote by~$Z$ the singular locus of~$S$ and by~$R$ the proper 
contour of a good projection; moreover, we denote by~$W$ the singular image, 
and by~$B$ the proper silhouette. For our purposes, the set-theoretic 
description of the contour is insufficient, so we define two scheme-theoretic 
notions. 

\begin{definition}
The \emph{fat contour}~$Y$ is the one-dimensional scheme defined by the equation 
of surface~$S$ and its derivative in the direction of the projection. This 
scheme is supported on the set~$Z\cup R$.

The \emph{fat silhouette}~$C$ is the one-dimensional scheme defined by the 
discriminant of the equation of the surface. This scheme is supported on the 
set~$W\cup B$.
\end{definition}
\begin{proposition} 
\label{proposition:generic_iso}
A good projection maps~$Y$ onto~$C$ and it is an isomorphism except over the 
special points of~$C$.
\end{proposition}
\begin{proof}
Since the projection is good, it is injective except over the 
special points. The component of~$Y$ supported on~$R$ is reduced because of the 
hypothesis that tangent lines through the center of projection intersect the 
surface with multiplicity~$2$ at contour points. Hence the set-theoretic 
isomorphism implies scheme-theoretic isomorphism for those points. This is not 
immediately the case for the component of~$Y$ supported on~$Z$. Locally at a 
smooth point of~$Z$ outside the contour, the surface~$S$ is analytically 
isomorphic\footnote{We can pass to the analytic category since the completion 
of a local Noetherian ring is faithfully flat, so it is enough to check the 
isomorphism property after passing to the completion (see the proof of 
\Cref{proposition:ideal_sheaf}).} to $z(z-x)=0$, the fat contour~$Y$ 
is defined by $2z-x=z(z-x)=0$, and $C$ is defined by $x^2 = 0$; hence the 
restriction of the projection to~$Y$ is an isomorphism (of schemes) with 
inverse described by the homomorphism of rings $\C[x,y,z]/(2z-x, z(z-x)) 
\longrightarrow \C[x,y]/(x^2)$ given by $([x],[y],[z]) \mapsto 
([x],[y],[x/2])$. 
\end{proof}

The strategy for reconstructing the fat contour of a good projection from the 
fat silhouette mimics the one in the smooth case. First of all, we express the 
sheaf~$\OO_Y(1)$, which provides the embedding of~$Y$ in~$\p^3$, as a twist of 
the dualizing sheaf~$\omega^{\circ}_Y$, which is a substitute in the non-smooth 
setting for the canonical sheaf. Using the ``upper shriek'' operation, in 
\Cref{lemma:dualizing,lemma:global_sections} we 
connect the dualizing sheaves of~$Y$ and~$C$, and obtain that in order to 
determine the direct image of~$\OO_Y(1)$ under a projection~$\pi$, it is 
enough to compute (a twist of) the sheaf $\HHom_{\OO_C} \bigl(\pi_{\ast} 
\OO_{Y}, \OO_C \bigr)$, which is supported at the special points of~$C$. The 
latter comes with a natural map to~$\OO_C$, and we show that this map is 
injective, proving that $\HHom_{\OO_C} \bigl(\pi_{\ast} \OO_{Y}, \OO_C 
\bigr)$ is an ideal sheaf. Therefore, the problem of determining a rational map 
sending~$C$ to~$Y$ becomes equivalent to the computation of the space of 
homogeneous forms of a certain degree that satisfy particular vanishing 
conditions at the special points of~$C$. This is analogous to the smooth 
situation, where we computed the adjoint forms of the silhouette.

Recall that a crucial step in the smooth situation is the fact that the 
contour~$R$ is linearly normal, or equivalently (for smooth varieties) that the 
standard map $\HH^0 \bigl( \p^3, \OO_{\p^3}(1) \bigr) \longrightarrow \HH^0 
\bigl( R, \OO_{Y}(1) \bigr)$ is an isomorphism. We prove that the latter 
condition holds also for the fat contour, which is very far from being smooth.

\begin{lemma}
 The map $\HH^0 \bigl( \p^3, \OO_{\p^3}(1) \bigr) \longrightarrow \HH^0 
\bigl( Y, \OO_{Y}(1) \bigr)$ is an isomorphism. In particular, $\HH^0 
\bigl( Y, \OO_{Y}(1) \bigr)$ is $4$-dimensional.
\end{lemma}
\begin{proof}
This follows from the fact that $Y$ is a complete intersection of two surfaces 
of degree~$d$ and~$d-1$, and so we have a graded free resolution of~$\OO_{Y}$ 
provided by the Koszul complex:
\[
 \resizebox{\textwidth}{!}{
 \xymatrix{
  0 \ar[r] & \OO_{\p^3}(-2d+1) \ar[r] &
  \OO_{\p^3}(-d) \oplus \OO_{\p^3}(-d+1) \ar[r] &
  \OO_{\p^3} \ar[r] & \OO_{Y} \ar[r] & 0
 }
 }
\]
Twisting by~$\OO_{Y}(1)$ and looking at the corresponding long exact sequence in 
cohomology yields the result.
\end{proof}

We now show how to reconstruct the fat contour~$Y$ and the projection 
\mbox{$\pi_{|_{Y}} \colon Y \longrightarrow C$} starting 
from the fat silhouette~$C$. As pointed out at the beginning of the section, 
this is carried out locally, and the local data are patched together using the 
fact that both schemes, being projective over a field, admit a \emph{dualizing 
sheaf} $\omega^{\circ}$ (see \cite[Proposition III.7.5]{Hartshorne1977}). 

In particular, in our case we have:
\begin{lemma}
\label{lemma:dualizing}
$\omega_Y^{\circ} \cong \OO_Y(2d-5)$ and $\omega_C^{\circ} \cong 
\OO_C(d^2-d-3)$.
\end{lemma}
\begin{proof}
 For a closed subscheme~$X$ of~$\p^n$ that is a local complete intersection of 
codimension~$r$, we have by \cite[Theorem III.7.11]{Hartshorne1977}
\[
 \omega_X^{\circ} \cong \omega_{\p^n} \otimes \bigwedge^r \bigl( \mscr{I} / 
\mscr{I}^2 \bigr)^{\vee} \,,
\]
where $\mscr{I}$ is the ideal sheaf of~$X$, and $(\cdot)^{\vee}$ denotes the 
dual sheaf. The claim follows from this formula and the 
definitions of~$Y$ and~$C$ as complete intersections.
\end{proof}

If we think of~$Y$ as an abstract scheme, it is embedded in~$\p^3$ via morphism 
determined by the global sections of the sheaf~$\OO_Y(1)$. Since our goal, as 
in the smooth situation, is to compute a map from~$C$ to~$\p^3$ whose image 
gives~$Y$, we link the global sections of~$\OO_Y(1)$ to the ones of a sheaf 
on~$C$.

\begin{lemma}
\label{lemma:global_sections}
 $\HH^0 \bigl( Y, \OO_Y(1) \bigr) \cong \HH^0 \Bigl( C, \HHom_{\OO_C} \bigl( 
\pi_{\ast} \OO_Y, \OO_C \bigr)(d^2-3d+3) \Bigr)$.
\end{lemma}
\begin{proof}
Since the projection 
$\pi_{|_{Y}} \colon Y \longrightarrow C$ is a finite affine morphism, we have 
that $\omega_Y^{\circ} = \pi^{!} \bigl( \omega_C^{\circ} \bigr)$ by 
\cite[Exercise III.7.2]{Hartshorne1977}. The sheaf $\pi^{!} \bigl( 
\omega_C^{\circ} \bigr)$, called ``$\pi$ upper shriek'' is defined in the 
following way (see \cite[Exercise 
III.6.10]{Hartshorne1977}). The sheaf $\HHom_{\OO_C} \bigl( \pi_{\ast} 
\OO_Y, \omega_C^{\circ} \bigr)$ is both an $\OO_C$-module and a 
$\pi_{\ast} \OO_Y$-module. For affine morphisms there is a correspondence 
between $\pi_{\ast} \OO_Y$-modules and $\OO_Y$-modules (see 
\cite[Exercise II.5.17e]{Hartshorne1977}); the $\OO_Y$-module 
corresponding to $\HHom_{\OO_C} \bigl( \pi_{\ast} \OO_Y, \omega_C^{\circ} 
\bigr)$ is defined to be~$\pi^{!} \bigl( \omega_C^{\circ} \bigr)$.
From \Cref{lemma:dualizing} we get
\begin{align*}
 \HH^0 \bigl( Y, \OO_Y(1) \bigr) &\cong \HH^0 \bigl( Y, 
\omega_Y^{\circ}(-2d+6) \bigr) \\
 &\cong \HH^0 \bigl( Y, \pi^{!}(\omega_C^{\circ})(-2d+6) \bigr) \\
 &\cong \HH^0 \Bigl( C, \pi_{\ast}\bigl(\pi^{!}(\omega_C^{\circ})(-2d+6)\bigl) 
\Bigr) \\
 &\cong \HH^0 \Bigl( C, \pi_{\ast}\bigl(\pi^{!}(\omega_C^{\circ})\bigl)(-2d+6) 
\Bigr) \, ,
\end{align*}
where the latter isomorphism is given by the \emph{projection formula} (see 
\cite[Exercise II.5.1d]{Hartshorne1977}). By analyzing the correspondence 
between $\pi_{\ast} \OO_Y$-modules and $\OO_Y$-modules as hinted in 
\cite[Exercise II.5.17e]{Hartshorne1977}, one sees that 
$\pi_{\ast}\bigl(\pi^{!}(\omega_C^{\circ})\bigl)$ is $\HHom_{\OO_C} \bigl( 
\pi_{\ast} 
\OO_Y, \omega_C^{\circ} \bigr)$ as an $\OO_C$-module. So we have
\begin{align*}
 \HH^0 \bigl( Y, \OO_Y(1) \bigr) &\cong \HH^0 \Bigl( C, 
\HHom_{\OO_C} \bigl( \pi_{\ast} \OO_Y, \omega_C^{\circ} \bigr)(-2d+6) \Bigr) \\
 &\cong \HH^0 \Bigl( C, \HHom_{\OO_C} \bigl( \pi_{\ast} \OO_Y, \OO_C(d^2-d-3) 
\bigr)(-2d+6) \Bigr) \\
 &\cong \HH^0 \Bigl( C, \HHom_{\OO_C} \bigl( \pi_{\ast} \OO_Y, \OO_C 
\bigr)(d^2-3d+3) 
\Bigr) \, .
\end{align*}
Notice that $\HHom_{\OO_C} \bigl( \pi_{\ast} \OO_Y, \OO_C \bigr)$ is 
supported at the singularities of~$C$ since by 
\Cref{proposition:generic_iso} a good projection is an isomorphism 
outside them.
\end{proof}

We are going to show that $\HHom_{\OO_C} \bigl( \pi_{\ast} \OO_Y, \OO_C \bigr)$ 
is an ideal sheaf. We then compute the graded part of degree~$d^2-3d+3$ of this 
ideal. By \Cref{lemma:global_sections}, a basis of this graded part 
provides a rational map from~$C$ to~$\p^3$ defined everywhere except at the 
special points (namely, the singularities of the silhouette); the image of this 
rational map is an open subscheme~$Y^\circ$ of~$Y$ intersecting both of its 
components nontrivially. The equation of the surface~$S$ is then the only 
polynomial of degree~$d$ vanishing on~$Y^\circ$ such that its derivative in 
the direction of the projection also 
vanishes on~$Y^\circ$, and this is what we compute in Algorithm 
\texttt{ReconstructGeneralSurface}.

\begin{proposition}
\label{proposition:ideal_sheaf}
 The $\OO_C$-module $\HHom_{\OO_C} \bigl( \pi_{\ast} \OO_Y, \OO_C \bigr)$ is an 
ideal 
sheaf.
\end{proposition}
\begin{proof}
There is a natural morphism of sheaves $\Phi \colon \HHom_{\OO_C} \bigl( 
\pi_{\ast} \OO_Y, \OO_C \bigr) \longrightarrow \OO_C$ sending a 
homomorphism~$\varphi$ to~$\varphi(1)$. We prove that $\Phi$ is injective, this 
showing that $\HHom_{\OO_C} \bigl( \pi_{\ast} \OO_Y, \OO_C \bigr)$ is an ideal 
sheaf. To do so, it is enough to show that for every closed point $c \in C$, the 
induced map $\Phi_c \colon \HHom_{\OO_C} \bigl( \pi_{\ast} \OO_Y, \OO_C \bigr)_c 
\longrightarrow \OO_{C, c}$ on stalks is injective. In turn, we can pass to the 
completion, namely we can tensor by~$\hat{\OO}_{C, c}$, and prove injectivity in 
that case, since the completion of a local Noetherian ring is faithfully flat 
(see \cite[\href{http://stacks.math.columbia.edu/tag//00MC}{Tag/00MC}, Lemma 
10.96.3]{stacks-project}). Since the formation of~$\Hom$ commutes with flat base 
change (see \cite[\href{http://stacks.math.columbia.edu/tag//087R}{Tag/087R}, 
Remark 15.60.20]{stacks-project}), it suffices to prove that the map
\[
 \hat{\Phi}_c 
\colon \Hom_{\hat{\OO}_{C,c}} \bigl( (\pi_{\ast} \OO_Y)_c \otimes 
\hat{\OO}_{C,c}, \hat{\OO}_{C,c} \bigr) \longrightarrow \hat{\OO}_{C,c} 
\]
is injective. Notice that the $\hat{\OO}_{C,c}$-module $(\pi_{\ast} \OO_Y)_c 
\otimes \hat{\OO}_{C,c}$ is isomorphic to the direct sum $\displaystyle 
\bigoplus_{y_i \colon 
\pi(y_i) = c} \hat{\OO}_{Y, y_i}$. In fact, by the Theorem on Formal Functions 
(see \cite[Theorem III.11.1 and Remark III.11.1.2]{Hartshorne1977}) we have that
\[
 (\pi_{\ast} \OO_Y)_c \otimes \hat{\OO}_{C,c} 
 \; \cong \;
 \HH^0 \bigl( \hat{Y}, \OO_{\hat{Y}} \bigr) \, .
\]
where $(\hat{Y}, \OO_{\hat{Y}})$ is the \emph{completion of~$Y$ 
along~$\pi^{-1}(c)$} (see \cite[Definition III.9.3]{Hartshorne1977}). As a 
topological space, $\hat{Y}$ is just~$\pi^{-1}(c)$, so in our case it is a 
finite union of points (namely, the closed points $y_i \in Y$ such that 
$\pi(y_i) = c$), so the group of global sections of its structure sheaf is the 
direct sum of the groups of sections on each of these points. For any closed 
point $y_i \in \pi^{-1}(c)$, the group~$\HH^0 \bigl( y_i, \OO_{\hat{Y}} 
\bigr)$ is~$\hat{\OO}_{Y, y_i}$: in fact, by definition $\HH^0 \bigl( y_i, 
\OO_{\hat{Y}} \bigr)$ is the limit $\displaystyle \lim_{\leftarrow_n} \bigl( 
\OO_{Y, y_i} / J_i^n \cdot \OO_{Y, y_i} \bigr)$, where $J_i$ is the ideal of 
$\pi^{-1}(c)$ at~$y_i$. Since the radical of~$J_i$ is the maximal ideal 
of~$\OO_{Y, y_i}$, the two ideals define the same topology (see \cite[end of 
Section~III.2.5]{Bourbaki1972}), and so $\displaystyle \lim_{\leftarrow_n} 
\bigl( \OO_{Y, y_i} / J_i^n \cdot \OO_{Y, y_i} \bigr) = \hat{\OO}_{Y, y_i}$. 
Hence, we just need to prove that 
\[
 \Hom_{\hat{\OO}_{C,c}} 
 \left( 
  \bigoplus_{y_i \colon \pi(y_i) = c} \hat{\OO}_{Y, y_i}, \,
  \hat{\OO}_{C,c} 
 \right) \longrightarrow \hat{\OO}_{C,c}
\]
is injective for every closed point $c \in C$. Notice that for every closed 
point~$c$ such that $\pi_{|_{\pi^{-1}(c)}} \colon \pi^{-1}(c) \longrightarrow 
\{c\}$ is an isomorphism, there is nothing to prove. Hence, the only points we 
need to care about are the seven types of special points. The statement then 
follows from \Cref{lemma:total_fractions}, which describes a sufficient 
condition for injectivity, and \Cref{lemma:iso_nonzerodivisors}, which 
proves that this condition is met for each of the seven possible special points.
\end{proof}

\begin{lemma}
\label{lemma:total_fractions}
Let $E$ be a ring with total fraction ring $K$, namely $K$ is the localization
of $E$ at the set $R(E)$ of non-zerodivisors. Let $F$ be a subring of $K$
containing~$E$. 
Then the homomorphism of $E$-modules $\alpha \colon \Hom_E(F,E) \longrightarrow 
E$ sending~$\varphi$ to~$\varphi(1)$ is injective, and its image equals
the \emph{conductor ideal}
\[ 
 \{w \in E \, \colon \, wF \subset E \} \,.
\]
\end{lemma}
\begin{proof}
We first show that if $\varphi\in\Hom_E(F,E)$, then $\varphi(f) = 
\varphi(1)\cdot f$ for all $f\in F$. We localize $\varphi\colon F 
\longrightarrow E$ at $R(E)$ and obtain $\tilde{\varphi} \colon K 
\longrightarrow K$. The latter fulfills $\tilde{\varphi}(f) = 
\tilde{\varphi}(1)\cdot f$ because it is $K$-linear. Since $\varphi$ is the 
restriction of~$\tilde{\varphi}$ to~$F$, and both~$f$ and~$1$ are in~$F$, the 
claim is established.

If $\varphi(1)=0$, then by what we proved $\varphi=0$, and thus injectivity 
of~$\alpha$ is established.

Let $\varphi\in\Hom_E(F,E)$. For each $w\in F$, we have $\alpha(\varphi)\cdot 
w=\varphi(1) \cdot w=\varphi(w)\in E$, hence $\alpha(\varphi)$ is in the 
conductor ideal. Conversely, if $u$ is in the conductor ideal, then we define 
$\varphi \colon F \longrightarrow E$ as $w \mapsto u\cdot w$. Then 
$\alpha(\varphi)=u$. Hence the image of the map~$\alpha$ coincides with the 
conductor ideal.
\end{proof}

\begin{lemma}
\label{lemma:iso_nonzerodivisors}
 Let $c \in C$ be a closed point of~$C$ that is a special point for the fat
silhouette. Set $E = \hat{\OO}_{C, c}$ and $\displaystyle F = \bigoplus_{y_i 
\colon \pi(y_i) = c} \hat{\OO}_{Y, y_i}$. Then the homomorphism $E 
\longrightarrow F$ induced by the projection~$\pi$ is injective and becomes an isomorphism when 
we localize by the non-zerodivisors of~$E$.
\end{lemma}
\begin{proof}
 The statement follows from \Cref{proposition:generic_iso}. In fact, 
the statement holds if we prove that we obtain an isomorphism after localizing 
by a single non-zerodivisor. Geometrically the latter is true if and only if 
the projection~$\pi$ defines an isomorphism between the distinguished open set 
defined by the non-zerodivisor, and its preimage under~$\pi$. In view of 
\Cref{proposition:generic_iso}, it is then enough to show that for 
every special point $c \in C$ there is a non-zerodivisor in~$E$ vanishing on 
the special point. Since $E$ can be brought to the form $\C\pp{x,y}/(h)$ 
for a bivariate power series~$h$, it is enough to show that there always exists 
a non-zerodivisor in the ideal~$([x],[y])$ of~$E$. This is true since every 
zerodivisor of~$E$ correspond to a factor of~$h$, and since we have infinitely 
many different elements in~$([x],[y])$ of the form $[x + \lambda y]$ for 
$\lambda \in \C$, it is always possible to choose $\bar{\lambda}$ so that $x + 
\bar{\lambda} y$ is not a factor of~$h$.
\end{proof}

The proof of \Cref{proposition:ideal_sheaf} is complete and we know that 
the $\OO_C$-module $\HHom_{\OO_C} \bigl( \pi_{\ast} \OO_Y, \OO_C 
\bigr)$ is an ideal sheaf. Next, we compute the image of the map
\begin{equation}
\label{eq:conductor_ideal}
 \HHom_{\OO_C} \bigl( \pi_{\ast} \OO_Y, \OO_C \bigr)_c \otimes \hat{\OO}_{C,c} 
\longrightarrow \hat{\OO}_{C,c} \,, 
\end{equation}
namely the completions of the stalks of this ideal sheaf, 
for every special point~$c \in C$.

\Cref{conductor} explains how one can compute the image of the 
map~\eqref{eq:conductor_ideal} given a local equation of the surface~$S$ at a 
special point~$c$. In the next paragraph we clarify how we can compute, 
starting from these local data, the sections of a twist of the ideal 
sheaf~$\mscr{I}$ which is the image of $\HHom_{\OO_C} \bigl( \pi_{\ast} \OO_Y, 
\OO_C \bigr)$ in~$\OO_C$. From the discussion above, the global sections of 
$\mscr{I}(d^2 - 3d + 3)$ provide the map sending the fat silhouette~$C$ to the 
fat contour~$Y$.

Notice that, as we already proved in \Cref{reconstruction_smooth}, the 
global sections of~$\mscr{I}(d^2 - 3d + 3)$ are homogeneous polynomials of 
degree~$d^2 - 3d + 3$ satisfying particular properties. A homogeneous 
polynomial~$F$ of degree~$e$ is a global section of~$\mscr{I}(e)$ if and only 
if for any special point~$c$ the localization of~$F$ at~$c$ is in the 
stalk~$\mscr{I}_c$. The set of polynomials in~$\C[x,y,z]$ such that their 
localization at a point~$c$ belongs to~$\mscr{I}_c$ is a homogeneous ideal. The 
intersection of all these ideals provides the ideal defining~$\mscr{I}$.
Therefore, using the formulas provided in \Cref{conductor} we can 
compute all these ideals for every special point $c \in C$. 

The formula for the conductor ideal of a transversal intersection of~$B$ and~$W$
is not the same for the two possible types of these special points: if the 
transversal intersection is coming from a pinch point, then the conductor ideal 
is trivial, while if the intersection is the projection of of two distinct 
points, one in~$R$ and one in~$Z$, then the ideal is the sum of the square of 
the ideal of~$W$ and the ideal of~$B$ at the point. We could not find of a way 
to tell the two cases apart given only the equations of~$B$ and~$W$. It is of 
course possible to try out each of the finitely many cases, compute the result, 
and check it by comparing the discriminant with the given polynomial (in most 
cases, the computation will terminate with an error because the dimension of 
some vector space is not as expected). 

This concludes the explanation of the correctness of Algorithm 
\texttt{ReconstructGeneralSurface}.

\begin{algorithm}
\caption{\texttt{ReconstructGeneralSurface}}
\begin{algorithmic}[1]
  \Require A curve $C \subset \p^2$ with simple and double components.
  \Ensure A surface~$S \subset \p^3$ with ordinary singularities together with 
a good projection to~$\p^2$ such that $C$ is the fat silhouette of this 
projection, if such a surface exists; an error otherwise.
  \Statex
  \State {\bfseries Compute} the special points of the fat silhouette (see \Cref{remark:special_points}).
  \State {\bfseries Choose} a subset of the transversal intersections between 
proper silhouette and singular image to be considered as images pinch points.
  \For{each special point}
  \State {\bfseries Compute} the ideal whose localization at 
the special point coincides with the conductor ideal. Use the equivariant 
formulas given in \Cref{lemma:conductor_invariant} and subsequent 
discussion. 
  \State {\bfseries Homogenize} the ideal.
  \EndFor
  \State {\bfseries Intersect} all these ideals. Let $K$ be the result.
  \State {\bfseries Select} in $K$ a form $G_1$ of degree $d^2-3d+2$.
  \State {\bfseries Select} in $K$ a form $G_2$ of degree $d^2-3d+3$
which is not a multiple of~$G_1$.
  \State {\bfseries Compute} the ideal~$I$ of the image~$R$ of~$B$ under the
map \[ (x:y:z) \mapsto (x:y:z:G_2/G_1). \]
  \State {\bfseries Select} in $I$ a form~$H$ of degree~$d-1$.
  \State {\bfseries Select} in $I$ a form~$F$ whose derivative is a scalar
multiple of $H$.
  \State \Return $F$ if its discriminant is the fat silhouette; \texttt{fail} 
otherwise.
\end{algorithmic}
\end{algorithm}

\begin{remark}
\label{remark:special_points}
In our implementation, in order to determine the special points of the fat 
silhouette and to sort them by their type, we do as follows. We factor the 
equation of the fat silhouette as $U^2 V$, where $U$ is the equation of the 
singular image and $V$ is the equation of the proper silhouette. We 
then consider a general projection $\p^2 \dashrightarrow \p^1$ and we compute 
the discriminant of both~$U$ and~$V$ with respect to this projection and the 
resultant of~$U$ and~$V$ with respect to this projection. In this way, 
depending on the multiplicities of the corresponding factor in the 
discriminants or in the resultant, we are able to distinguish the various 
types of special points.
\end{remark}

To further comment Algorithm \texttt{ReconstructGeneralSurface} in 
\Cref{remark:uniqueness}, we introduce an equivalence relation between 
surfaces in~$\p^3$.

\begin{definition}
Let $S_1, S_2 \subset \p^3$ be two surfaces not passing through a 
point~$p$.
We say that $S_1$ is \emph{equivalent} to~$S_2$ if and only if there is a 
projective automorphism of~$\p^3$ that fixes all lines through~$p$ and that 
maps~$S_1$ to~$S_2$. Note that the equations of equivalent surfaces have the 
same discriminant with respect to~$w$, up to scaling. In other words, the 
surfaces~$S_1$ and~$S_2$ are equivalent over their silhouette.
\end{definition}

\begin{remark}
\label{remark:uniqueness}
For each choice of pinch points, the selection of the form~$G_1$ in Step~$8$ 
is unique up to scaling, the selection of the form~$G_2$ in Step~$9$ is unique 
up to scaling and up to multiples of~$G_1$, and the choice of~$H$ and~$F$ 
is unique up to scaling. This makes the result unique up to equivalence.

By trying all possible choices of pinch points, the algorithm can be used to 
compute all possible surfaces with ordinary singularities whose discriminant 
locus coincides with the given curve up to equivalence.
\end{remark}

\begin{remark} 
\label{remark:front_back_view}
One might believe that equivalent surfaces ``look the same'' to a camera 
positioned at the center of projection, meaning that they give the same 
structure of hidden parts of the silhouette. This is not so, because the hidden 
part structure depends on the relative position of camera, surface, and plane 
at infinity.

Let us assume that there exists a hyperplane~$H$ through~$p$ that does not 
intersect the real part of the surface~$S$. Take coordinates~$x$, $y$, $z$, 
and~$w$ in~$\p^3$ so that $H$ is the plane $\{ z = 0 \}$ and $p = (0:0:0:1)$. 
In this way, the real part of~$S$ is contained in the affine space where $z 
\neq 0$. In affine coordinates, the projection from~$p$ is then given by $(X, 
Y, W) \mapsto (X,Y)$, where $X,Y,W$ are the de-homogenized coordinates. Then, 
there are exactly two different ways of defining hidden parts on the real 
points of~$S$: given two points~$q_1, q_2$ with the same $X$ and $Y$ 
coordinates, one says that $q_1$ is hidden by~$q_2$ (respectively, $q_2$ is 
hidden by~$q_1$) if the $W$-coordinate of $q_1$ is bigger (respectively, 
smaller) than the one of~$q_2$. We call the two hidden part structures obtained 
in this way the \emph{front view} and the \emph{back view} of the surface.

In \Cref{figure:front_rear} we show the front and the back view of the 
same surface, which exhibit different hidden part structures.
\begin{figure}
 \centering
 \begin{tabular}{ccc}
  \includegraphics[width=.3\textwidth]{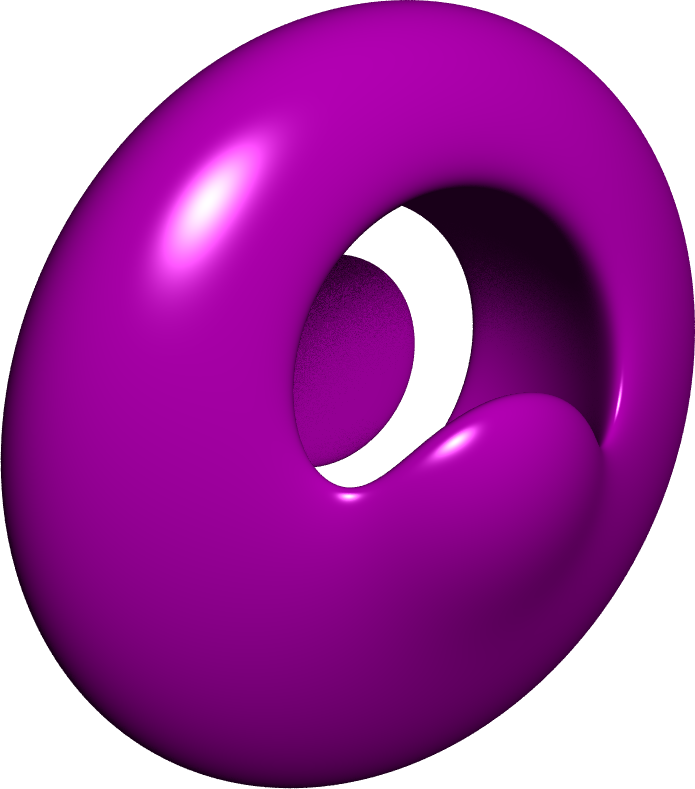} & &
  \includegraphics[width=.3\textwidth]{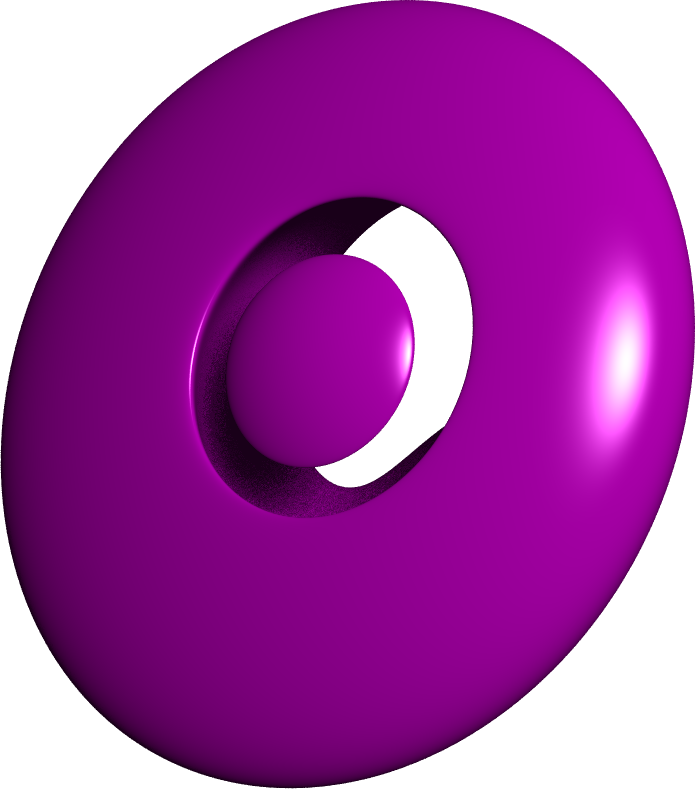}
 \end{tabular}
 \caption{The front and the back view of a quartic smooth surface.}
 \label{figure:front_rear}
\end{figure}
\end{remark}

We implemented the algorithm in Maple and tested it on a computer with an
Intel I7-5600 processor (1400 MHz). We report the timings in 
\Cref{table:timings}. The examples were surfaces of degree~$4$ and~$5$ 
with various types of singularities; the non-smooth cases are obtained by 
computing a random projection from a smooth model in a higher dimensional 
projective space.
The coefficients used in these random constructions were $5$ decimal digit
rational numbers chosen randomly. We projected the test surfaces to~$\p^2$
and used Algorithm \texttt{ReconstructGeneralSurface} to reconstruct them. Some 
of these test surfaces were ruled, and in this case we developed another 
algorithm --- which will be the subject of another paper --- that proves to be 
faster than the one presented here if we know a point on the proper silhouette.

As for the choice of the pinch points in Step~$2$, we took advantage of the fact
that our surfaces were defined over~$\Q$: we chose the conjugacy class of
points whose cardinality coincides with the known number of pinch points.

\begin{table}
 \caption{The table shows the degree of the surface~$S$, of the proper 
silhouette~$B$, and of the singular image~$W$; then the number of nodes and 
cusps of~$B$, the number of nodes and triple points of~$W$, the number of 
tangential intersections, pinch points, and other transversal intersection 
points, and the computing time in CPU seconds.}
\begin{tabular}{cccccccccccc}
 \toprule
$d$ & $B$ & $W$ & $n(B)$ & $c(B)$ & $n(W)$ & $t(W)$ & $t$ & $p$ & $o$ & time & type \\
\midrule
4 & 8 & 2 & 8 & 12 & 1 & 0 & 0 & 8 & 4 & 12s & ruled (elliptic base) \\
4 & 8 & 2 & 4 & 12 & 0 & 0 & 4 & 4 & 4 & 6s & Del Pezzo \\
4 & 6 & 3 & 4 & 6 & 1 & 0 & 2 & 4 & 6 & 4s & ruled \\
4 & 12 & 0 & 12 & 24 & 0 & 0 & 0 & 0 & 0 & 5s & smooth \\
4 & 6 & 3 & 0 & 9 & 0 & 1 & 6 & 6 & 3 & 3s & Veronese \\
5 & 20 & 0 & 60 & 60 & 0 & 0 & 0 & 0 & 0 & 180s & smooth \\
5 & 10 & 5 & 12 & 18 & 3 & 1 & 18 & 8 & 12 & 400s & Del Pezzo \\
5 & 8 & 6 & 12 & 9 & 6 & 1 & 12 & 6 & 15 & 130s & ruled \\
\bottomrule
\end{tabular}
 \label{table:timings}
\end{table}

\begin{example}[Quartic Del Pezzo surface] 
\label{example:del_pezzo}
A general projection of a smooth quartic Del Pezzo surface in~$\p^4$ to~$\p^3$ 
is a quartic surface~$S$. The singular curve~$Z$ of~$S$ is an irreducible conic. 

The proper silhouette is an octic curve. To produce a concrete example, we 
start with the surface~$S$ with the following equation:
\[ 
 (x^2+y^2-w^2)^2 - z^2(x^2-y^2)+z^4 = 0 \,. 
\]
The singular locus~$Z$ is given by the conic
\[
 z = x^2 + y^2 - w^2 = 0 \,.
\]
By projecting from the point $\left( \frac{2987}{918} : \frac{58}{33} : 
\frac{29}{6} : 1 \right)$, the singular locus is mapped isomorphically to the 
plane conic~$W$ with equation $x^2 + y^2 - w^2=0$. 
The proper silhouette~$B$ is an octic with 2 real components. 
The curves~$B$ and~$W$ intersect in 4 points tangentially and in 8 points 
transversally. Four of them are images of pinch points. We see only two of
the remaining $4$ because the other two are not real.

If we specify the correct pinch points, then the reconstruction algorithm 
returns the surface~$S$. If we specify the other four points as pinch points, 
then we obtain another surface~$S'$ that has only two real pinch points (see 
\Cref{figure:darboux}). Any other choice of pinch points does not give a 
surface.
\begin{figure}
 \centering
 \begin{tabular}{ccc}
  \includegraphics[width=.45\textwidth]{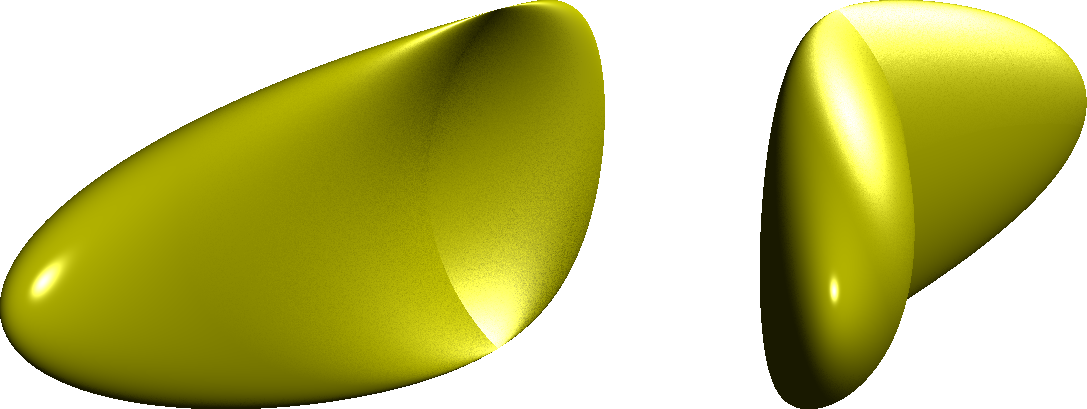} & \hspace{.5cm} & 
  \includegraphics[width=.45\textwidth]{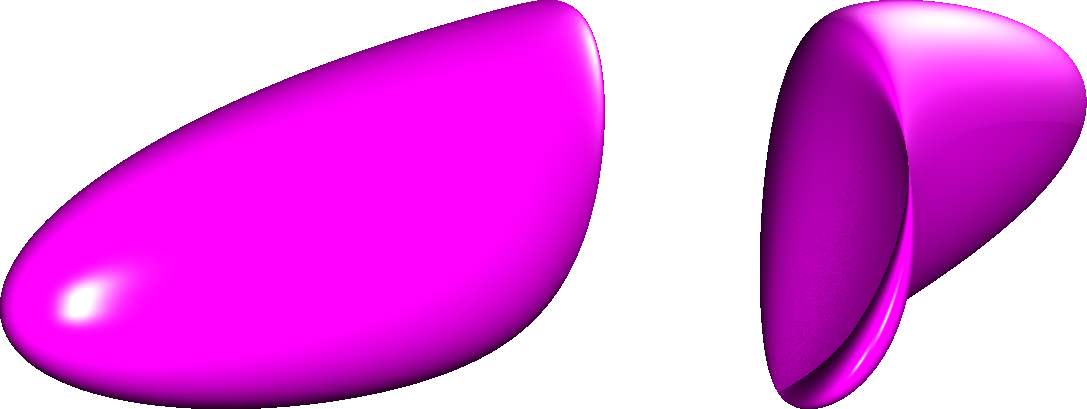} \\
  \includegraphics[width=.45\textwidth]{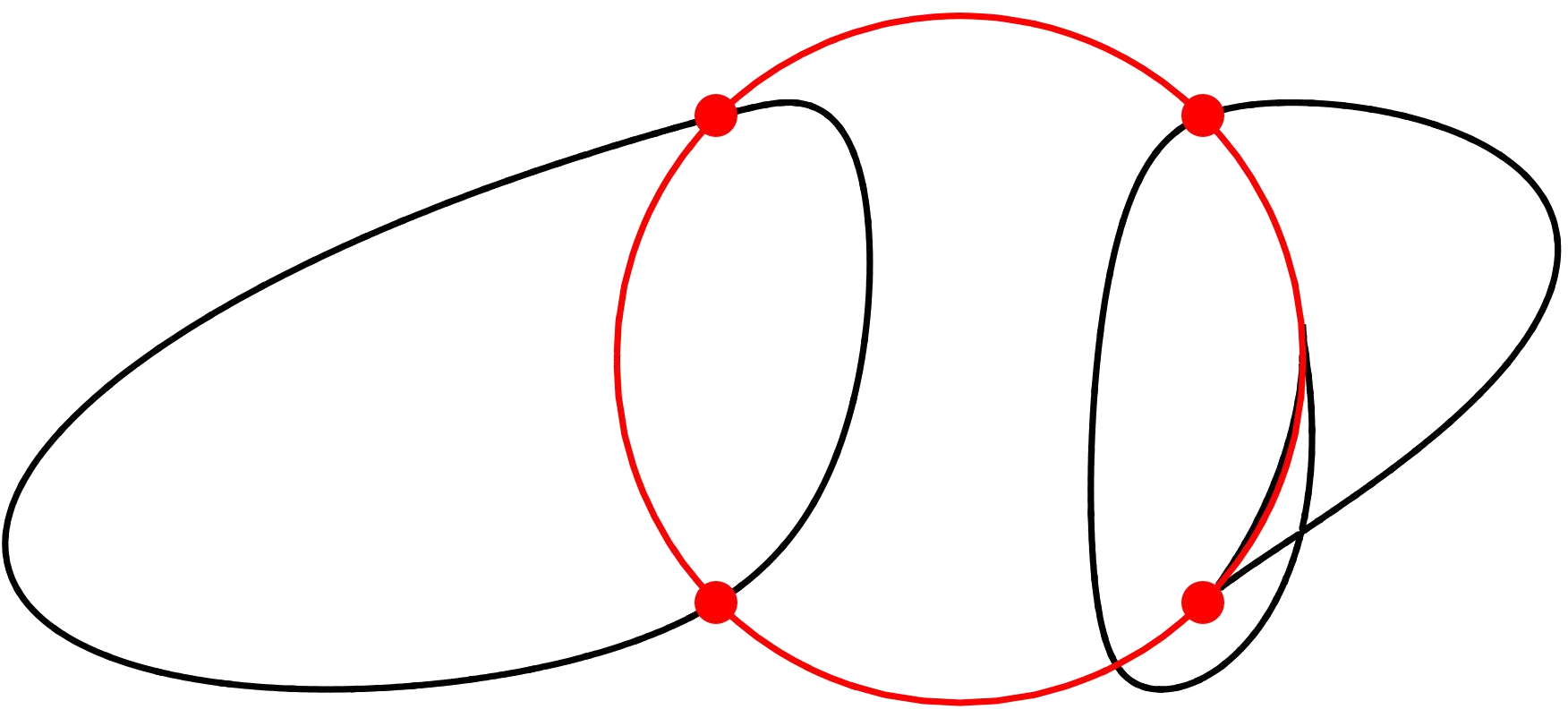} & & 
  \includegraphics[width=.45\textwidth]{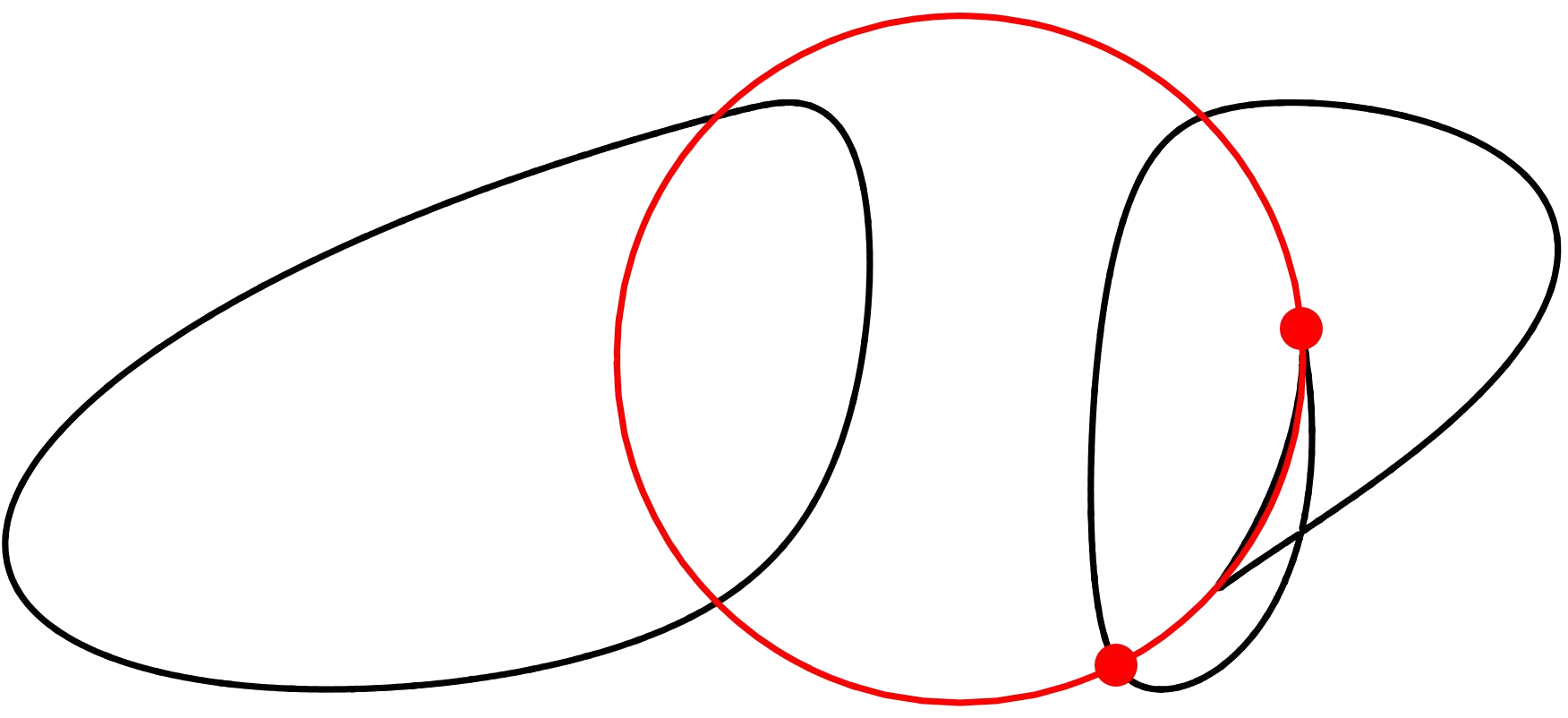} \\
  \includegraphics[width=.45\textwidth]{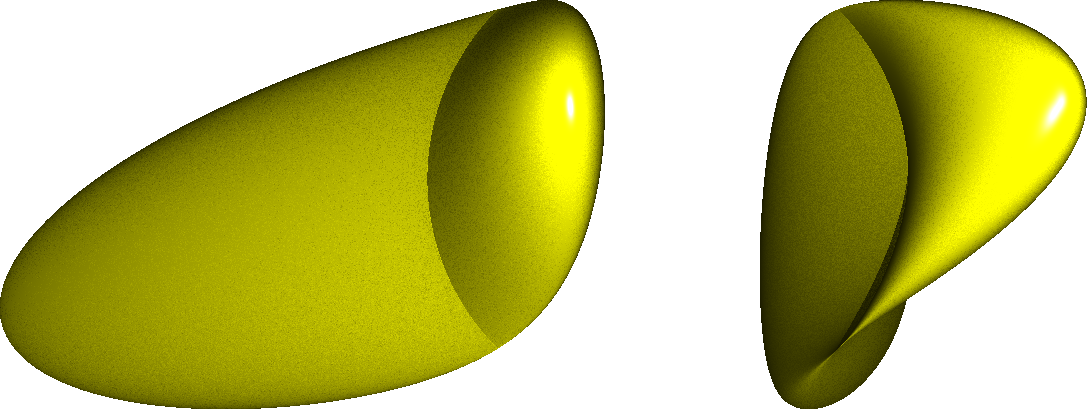} & & 
  \includegraphics[width=.45\textwidth]{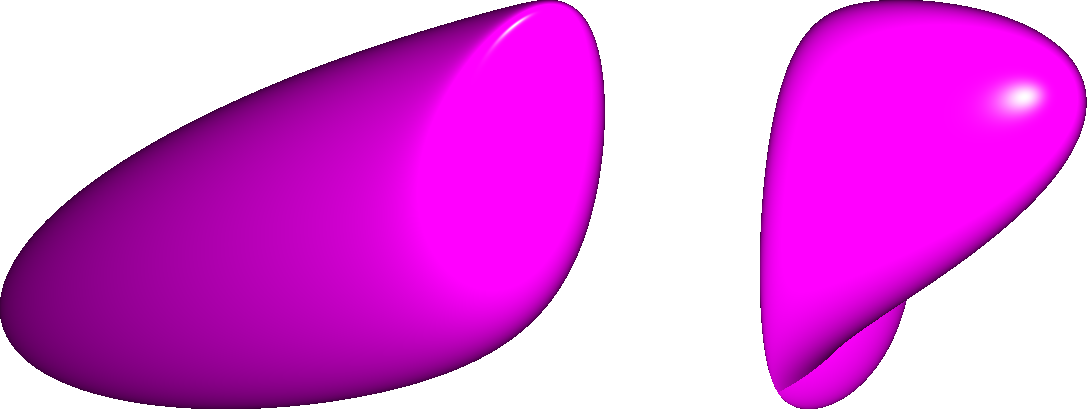} 
 \end{tabular}
 \caption{Two quartic surfaces with a singular conic with the same silhouette. 
The surface on the left, with front view, silhouette, and back view, has four 
real pinch points. The surface on the right has two real pinch points.}
 \label{figure:darboux}
\end{figure}
\end{example}

\begin{remark}
If a quadratic equation of the singular curve~$Z$ of a surface as in 
\Cref{example:del_pezzo} is positive definite, then by a change of 
coordinates we can suppose that $Z$ is given by
\[
 w = x^2 + y^2 + z^2 = 0 \, .
\]
Hence we see that if we fix a positive definite quadratic equation of~$Z$, we 
get a scalar product in the affine space~$\A^3$ obtained by removing the plane 
carrying~$Z$. This gives the space~$\A^3$ the structure of a Euclidean space; 
the conic~$Z$ is the \emph{absolute conic} with respect to this structure (by 
definition, this is a conic without real points in the plane at infinity). The 
reason for this setup is that, despite $Z$ has no real points, it can still be 
``seen'' in a photographic image obtained by central projection from a point $p 
\in \A^3$. The trick is to use a \emph{calibrated camera} (see 
\cite[Section~1.1]{Hartley2004}): if we mark the footpoint~$q$ of~$p$ on
the image plane and the intersection of this plane with a right circular cone 
with vertex~$p$ and axis through~$q$ and angle~$\frac{\pi}{4}$ (any other fixed 
angle would equally work), then all viewing angles~$\sphericalangle(q_1,p,q_2)$ 
for $q_1,q_2$ in the image plane can be computed by simple trigonometry. Hence 
the image plane is an elliptic plane, which means that we prescribe on it a 
conic without real points; in this case, this conic is the image of~$Z$ under 
the projection.

In this case, the two surfaces~$S_1$ and~$S_2$ that are obtained by reconstruction
are related by a spherical inversion with midpoint at the center of the 
projection. The reason for that is that the inversion of a quartic surface with 
the absolute conic as double curve is again a quartic surface with the absolute 
conic as double curve.
\end{remark}

\begin{example}[Veronese surface] 
\label{example:veronese}
The general projection of a Veronese surface is a quartic surface~$S$ with three 
singular lines $Z_1,Z_2,Z_3$ meeting in a triple point. Such a surface is called 
a \emph{Roman} or \emph{Steiner surface}, and is projectively equivalent to the
surface of equation\footnote{To obtain the isomorphism, move the three singular 
lines to the three axes; the ideal having the axis as double lines is generated 
by~$x^2y^2$, $x^2z^2$, $y^2z^2$ and~$xyz$; imposing that the surface has a 
triple point at the origin leads to the equation.}
\[ 
 x^2y^2 + x^2z^2 + y^2z^2 + xyzw = 0 \,.
\]
In this example, the three singular lines are the coordinate axes through the 
point~$(0\colon0\colon0\colon1)$. Each line contains two pinch points. The silhouette consists 
of three lines $W_1,W_2,W_3$ (the singular image) and a sextic~$B$ with $9$ 
cusps (the proper silhouette).
Each line~$W_i$, for $i=1,2,3$, is tangent to~$B$ at one point and 
intersects~$B$ transversally in~$4$ points. In order to recover the surface 
from the silhouette, we need to choose which are the projections of the $2$ 
pinch points on a line $Z_i$ among the four points of intersection 
between~$W_i$ and~$B$. There are $216$ possible cases.
\begin{figure}
 \centering
 \begin{tabular}{ccc}
  \includegraphics[width=.22\textwidth]{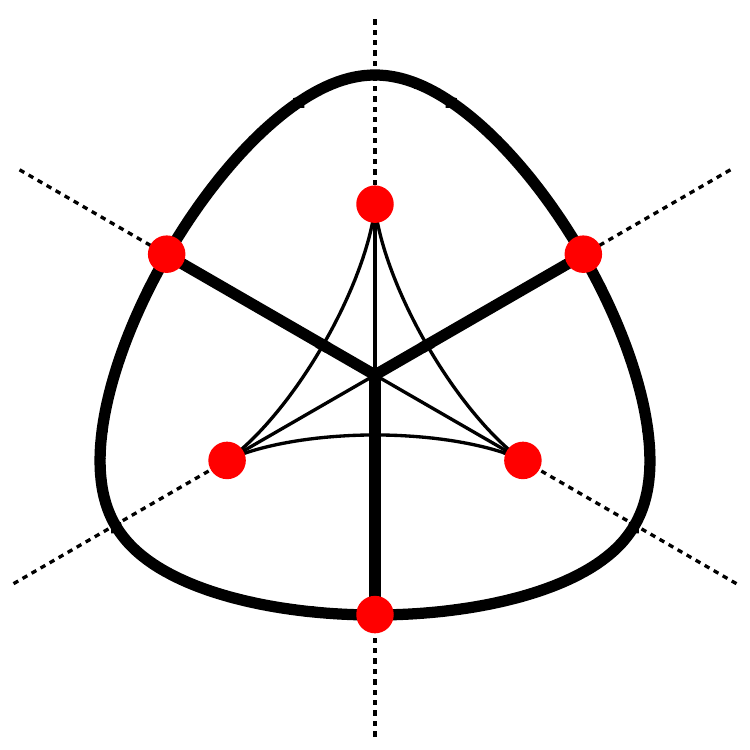} & \hspace{1cm} &
  \includegraphics[width=.22\textwidth]{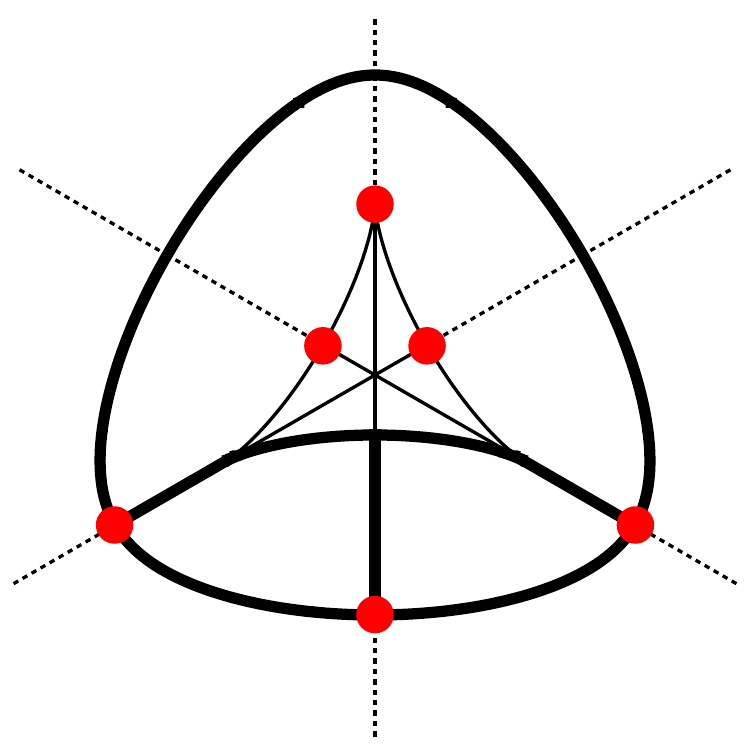} \\
  \includegraphics[width=.22\textwidth]{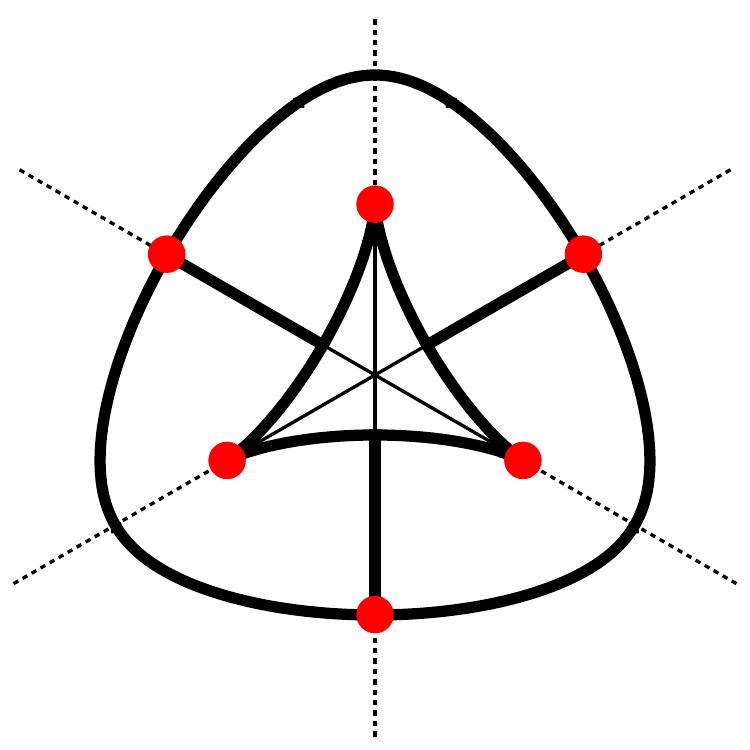} &&
  \includegraphics[width=.22\textwidth]{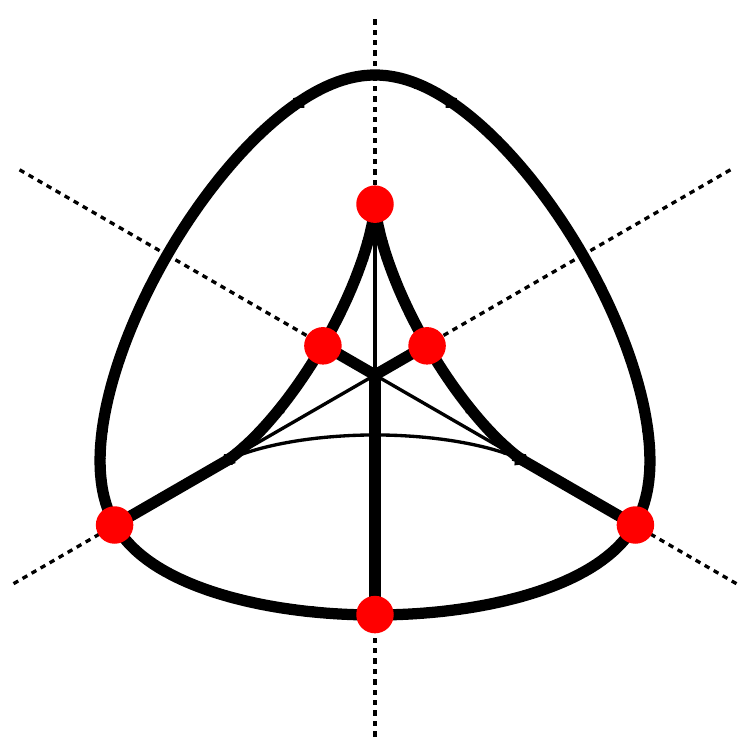} \\ 
  \hline
  \includegraphics[width=.22\textwidth]{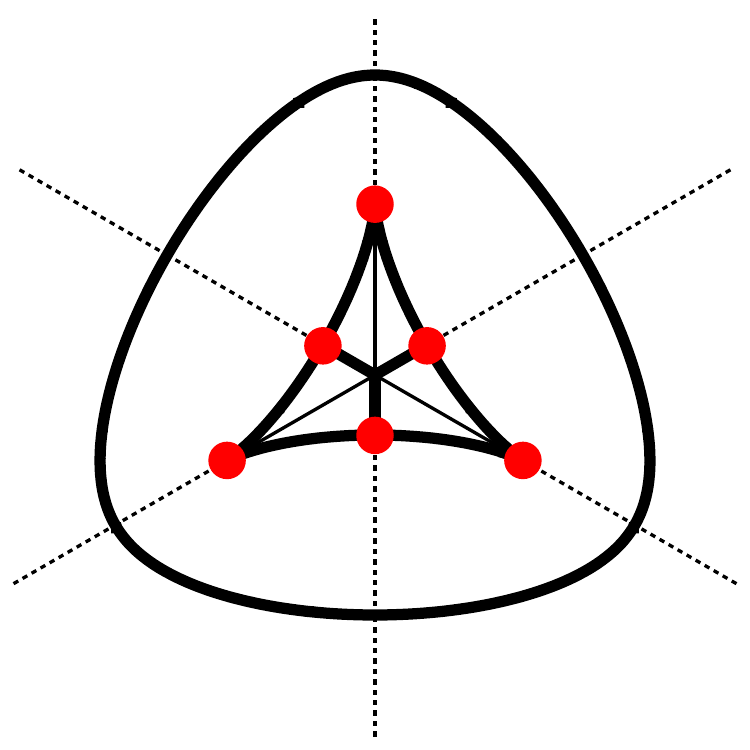} &&
  \includegraphics[width=.22\textwidth]{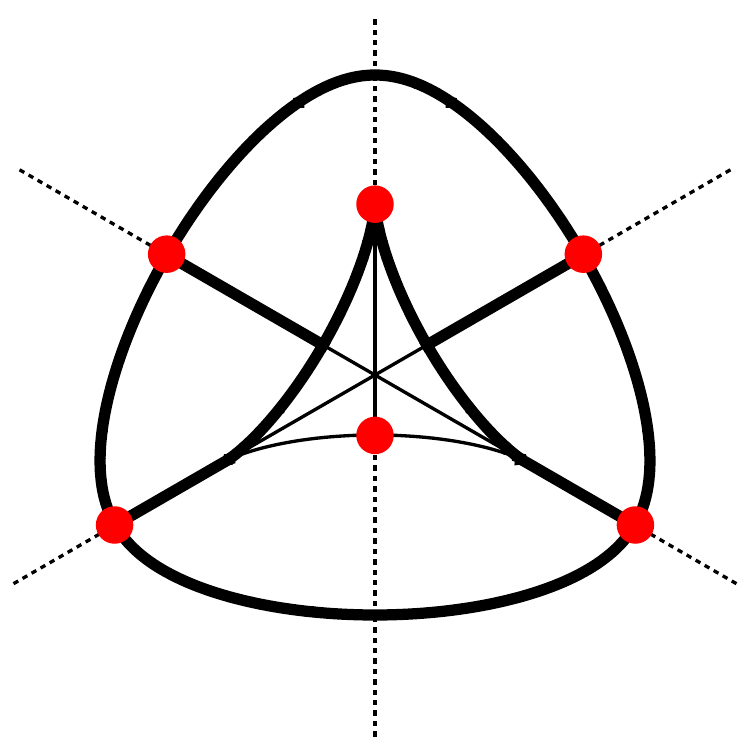} \\
  \includegraphics[width=.22\textwidth]{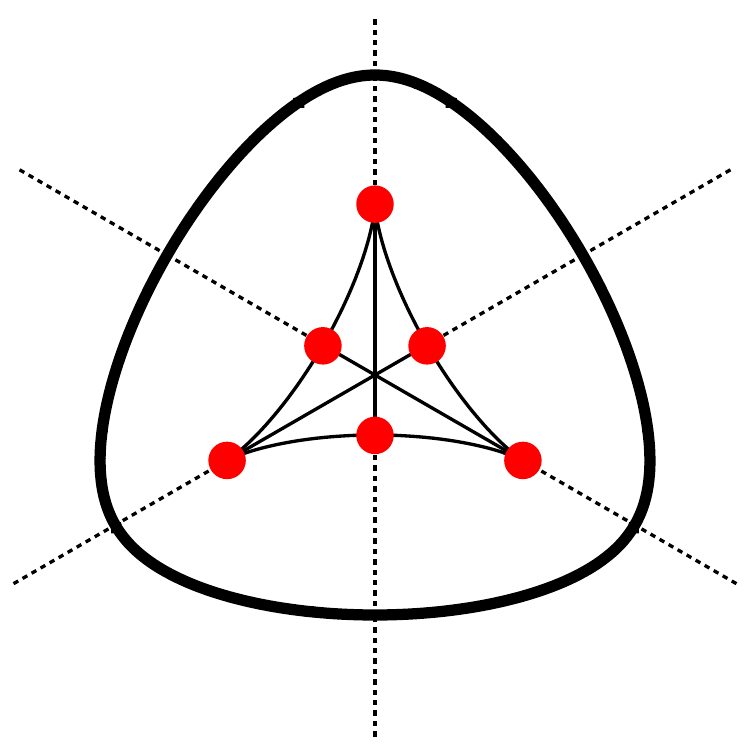} &&
  \includegraphics[width=.22\textwidth]{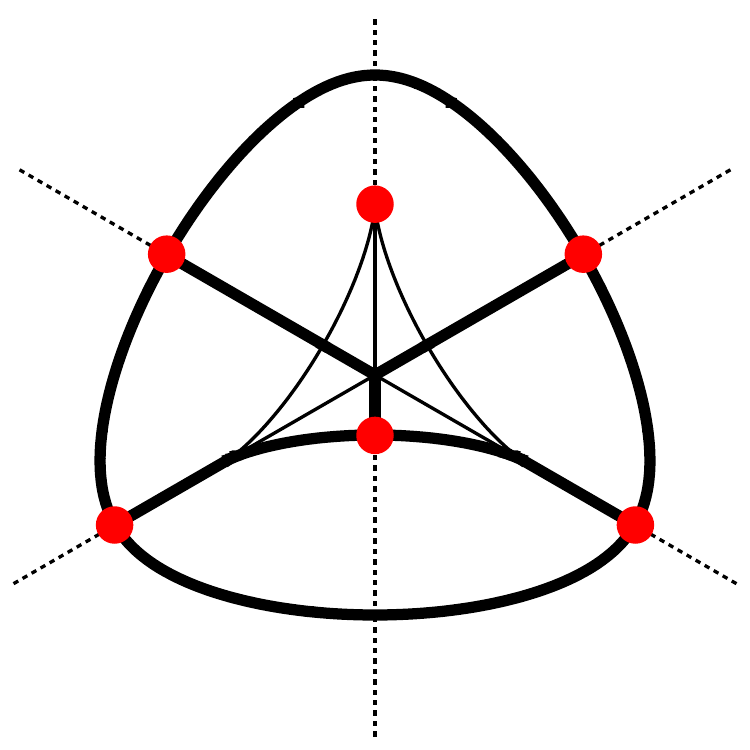} \\ 
  \hline
  \includegraphics[width=.22\textwidth]{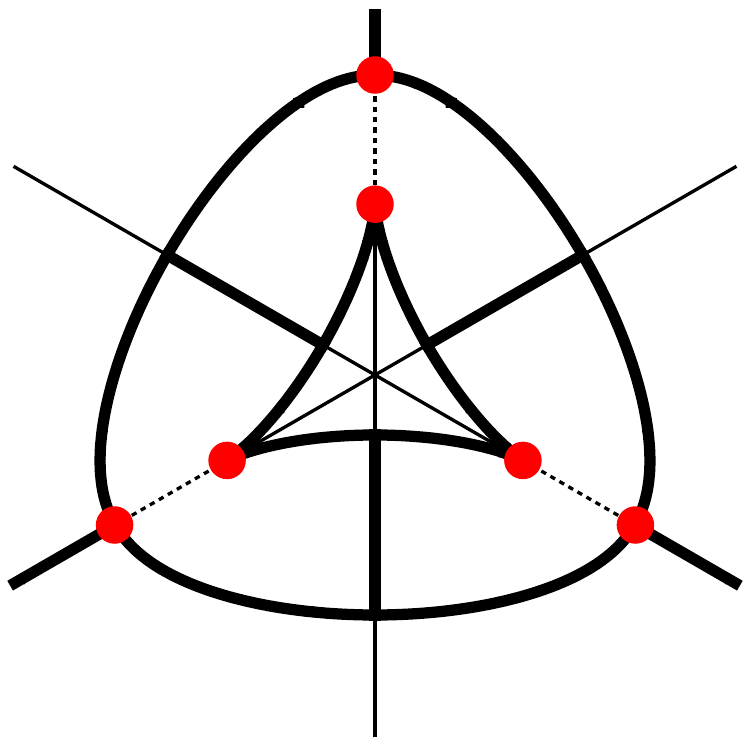} &&
  \includegraphics[width=.22\textwidth]{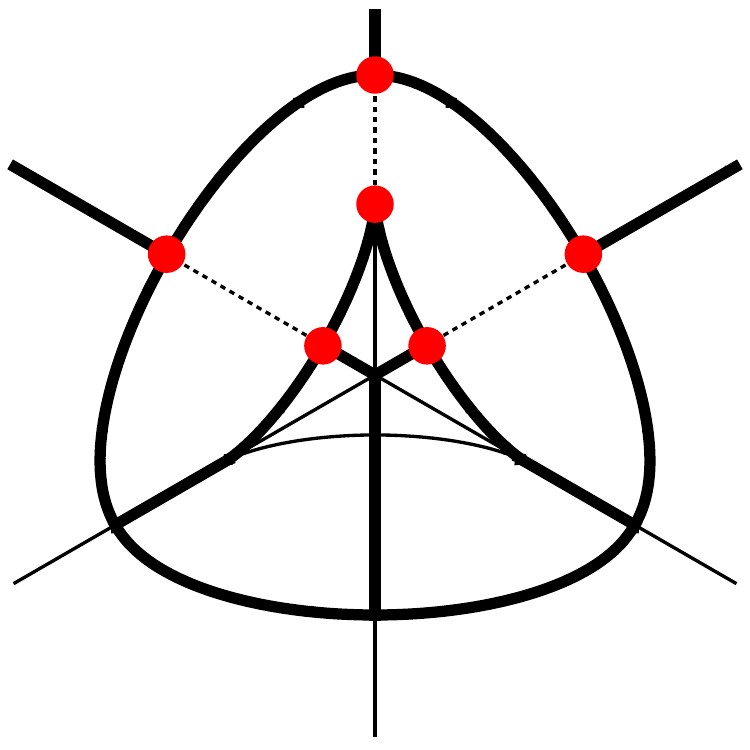} \\
  \includegraphics[width=.22\textwidth]{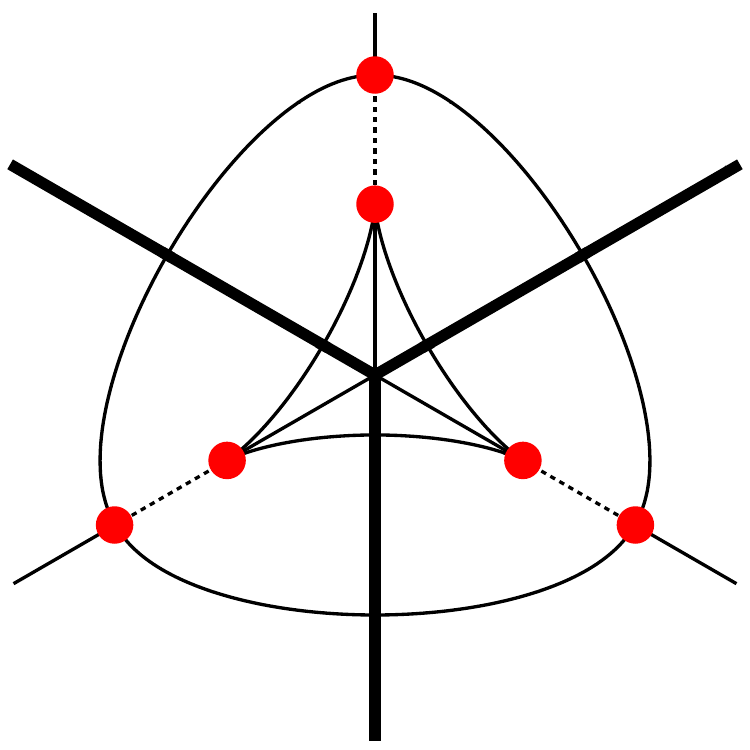} &&
  \includegraphics[width=.22\textwidth]{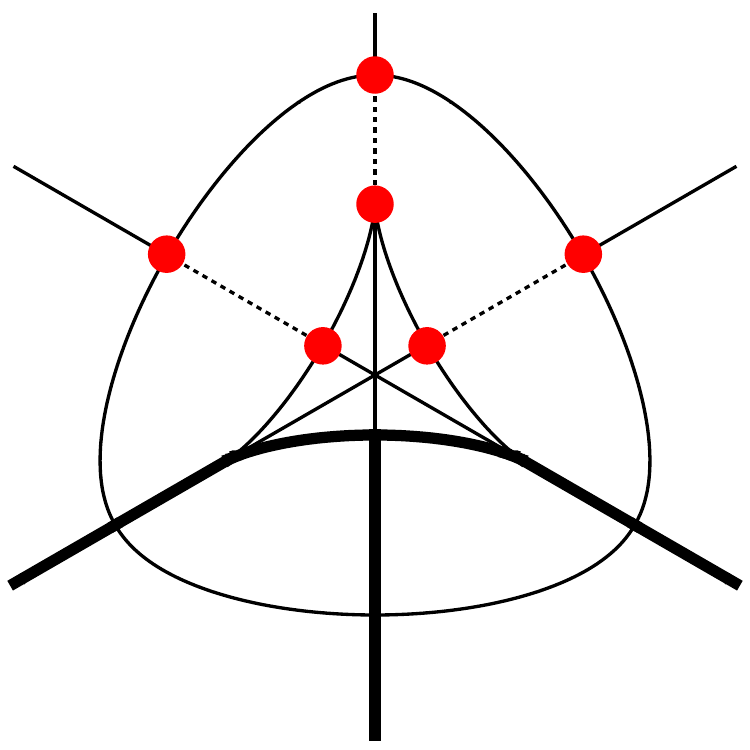}
 \end{tabular}
 \caption{These diagrams show the hidden parts and isolated lines of six 
non-equivalent Roman surfaces projecting to the same silhouette, front and back 
view. Six others can be obtained by rotating the three surfaces on the 
right by $120^{\circ}$ and $240^{\circ}$. Diagrams 
in the same double row are obtained by factorizing the same projection from the 
Veronese surface to~$\p^2$.}
 \label{figure:veronese}
\end{figure}
\begin{figure}
 \centering
 \begin{tabular}{ccc}
  \includegraphics[width=.22\textwidth]{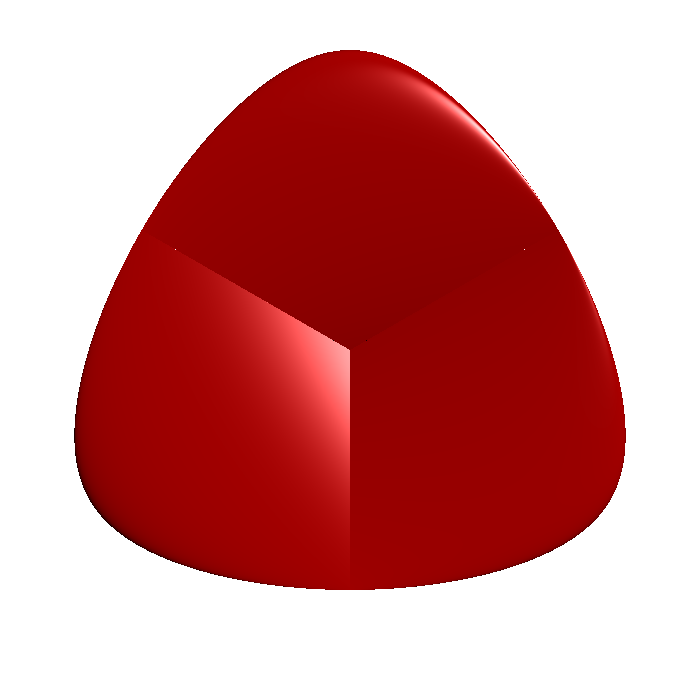} & \hspace{1cm} & 
  \includegraphics[width=.22\textwidth]{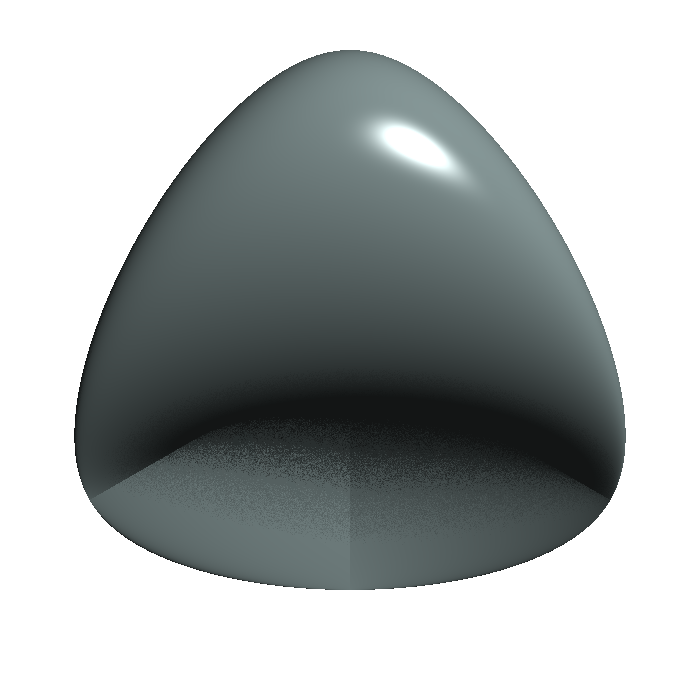} \\
  \includegraphics[width=.22\textwidth]{pictures/pb1.png} && 
  \includegraphics[width=.22\textwidth]{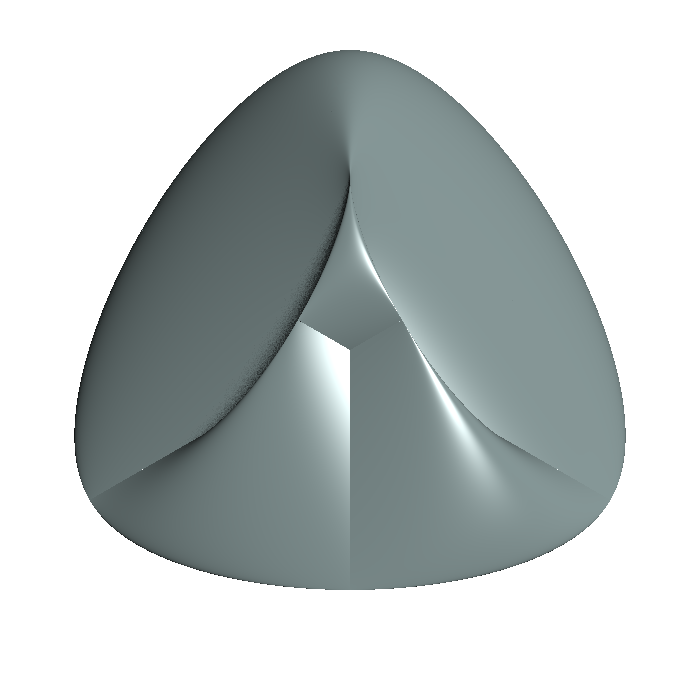} \\
  \includegraphics[width=.22\textwidth]{pictures/pf3.png} & \hspace{1cm} & 
  \includegraphics[width=.22\textwidth]{pictures/pf4.png} \\
  \includegraphics[width=.22\textwidth]{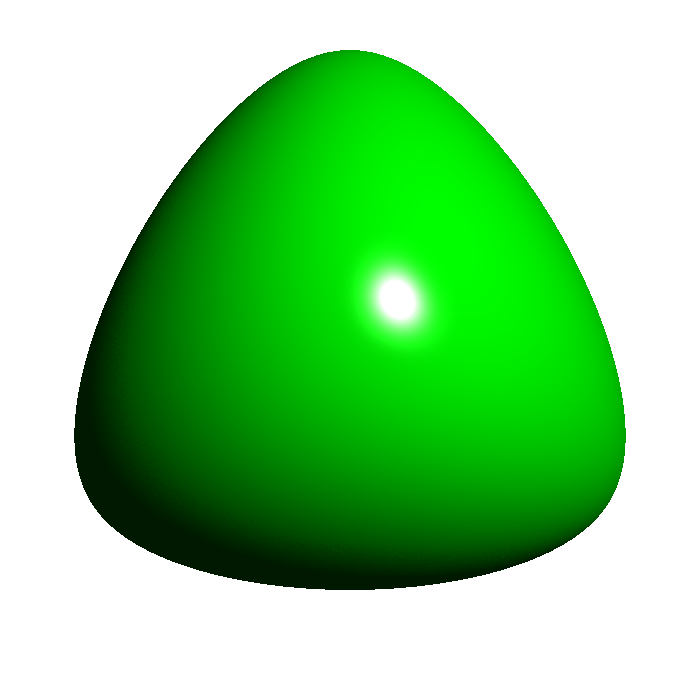} && 
  \includegraphics[width=.22\textwidth]{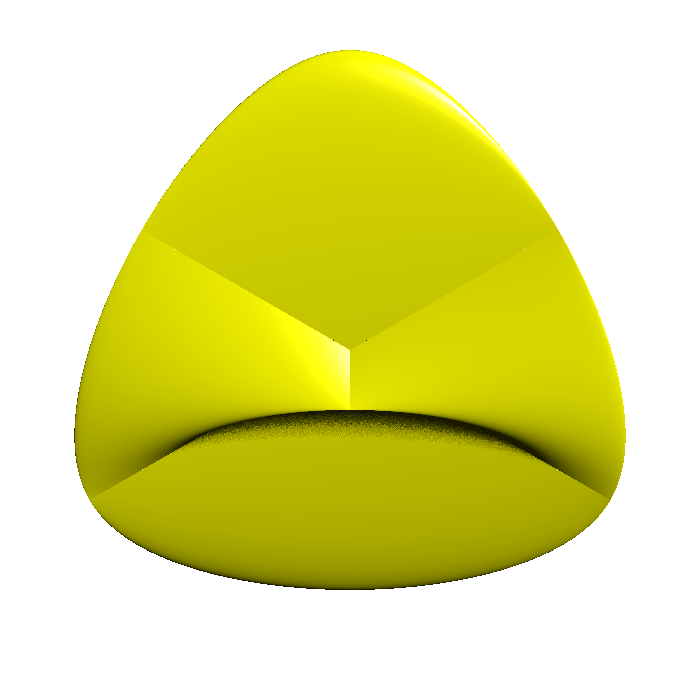} \\
  \includegraphics[width=.22\textwidth]{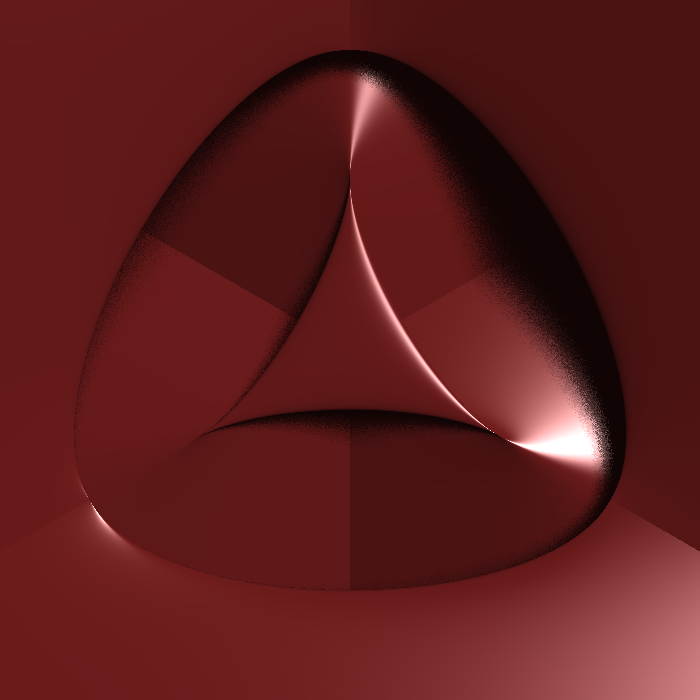} & \hspace{1cm} & 
  \includegraphics[width=.22\textwidth]{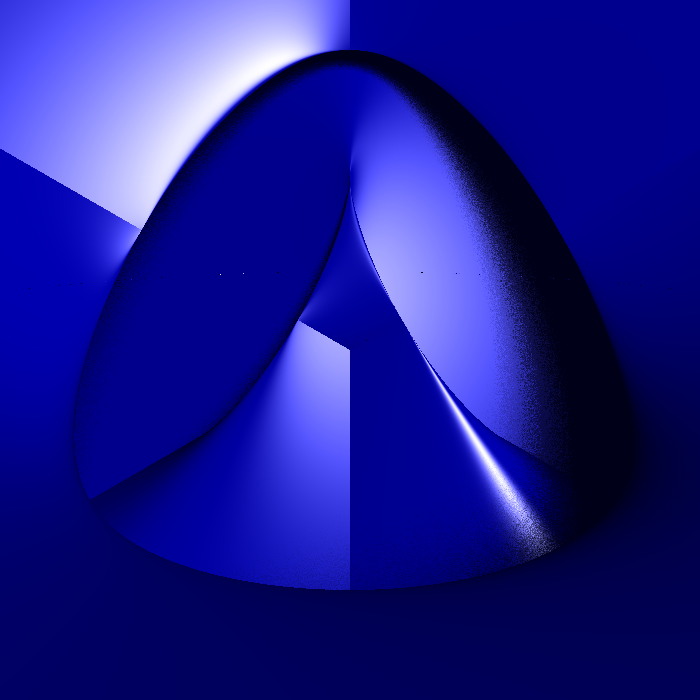} \\
  \includegraphics[width=.22\textwidth]{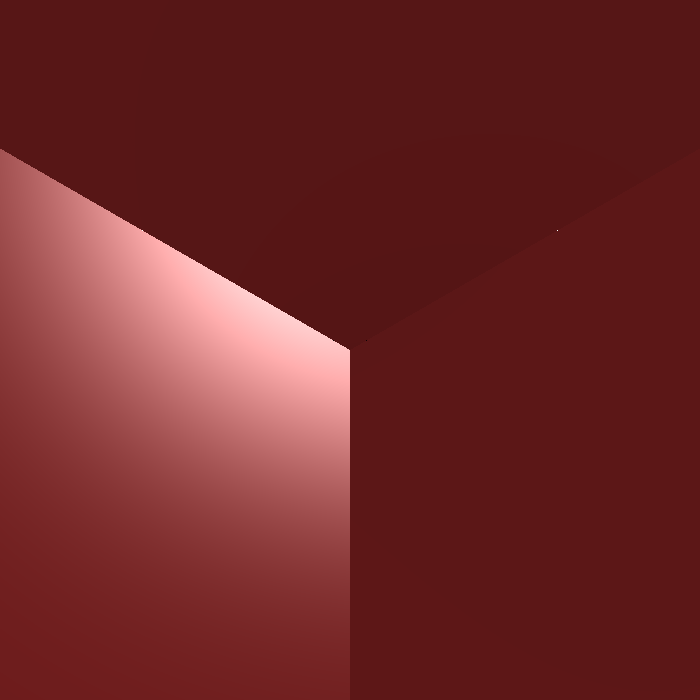} && 
  \includegraphics[width=.22\textwidth]{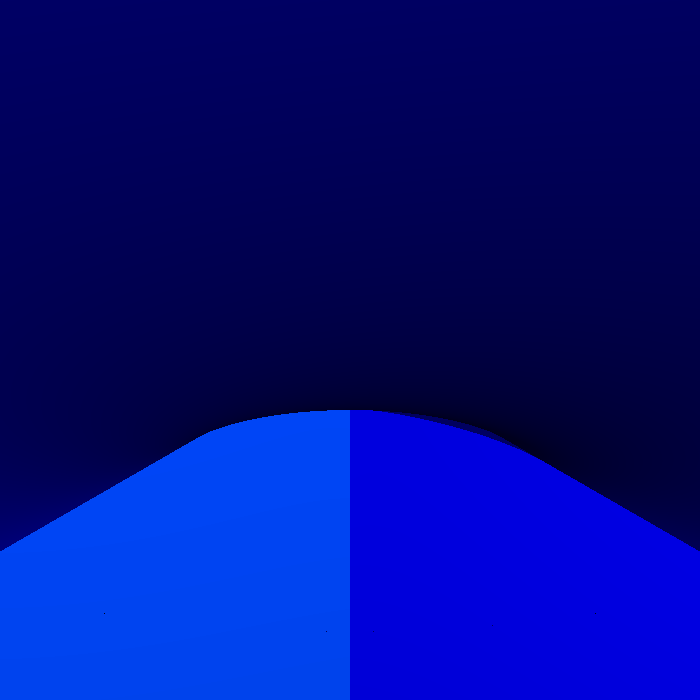} 
 \end{tabular}
 \caption{Here are six non-equivalent Roman surfaces with the same silhouette, 
front and back view, corresponding to the diagrams in 
\Cref{figure:veronese}.}
 \label{figure:veronese3D}
\end{figure}

The computation using our algorithm shows that $204$ choices lead to an error 
message, while $12$ choices lead to a Roman surface. Let us say that two such 
surfaces~$S_1$ and~$S_2$, both coming with a projection $f_i \colon S_i 
\longrightarrow \p^2$, are \emph{Veronese-equivalent} if there is a Veronese 
surface $V \in \p^5$ and projection maps $g_i \colon V \longrightarrow S_i$ such 
that $g_1 \circ f_1 = g_2 \circ f_2$. Then the 12 Roman surfaces are partitioned 
into three Veronese-equivalence classes, each consisting of four surfaces. The 
fact that there are three different ways to project a Veronese surface 
to~$\p^2$ for a fixed branching curve~$B$ has been found by Catanese, see 
\cite[Proposition 3.11]{Catanese1986}, improving an example of Chisini. The four 
different choices of factoring each of these three maps through a Roman surface 
are explained by the fact that that the preimage of the intersection point of 
the three lines $W_1,W_2,W_3$ consists of $4$ points in the Veronese surface, 
and three of them are mapped to the triple point of the Roman surface: there 
are four ways to choose a triple out of four points.

In \Cref{figure:veronese3D}, we show $6$ non-equivalent Roman surfaces 
with the same silhouette. They are divided in three groups, giving the three 
Veronese-equivalence classes. The diagrams in \Cref{figure:veronese} 
displays which parts of the silhouette are visible and which are hidden, and 
also which parts of the singular line are self-intersections and which are 
isolated lines. We see that for each Veronese-equivalence class we have an 
example where the visible/hidden structure is invariant under rotations 
by~$\pi/3$, and another which is not. By applying rotations to the 
non-invariant example, we get two more non-equivalent surfaces that are in the 
same Veronese-equivalence class. In this way we get all the $12$ 
non-equivalent Roman surfaces.

For four surfaces in \Cref{figure:veronese3D}, it is possible to find a 
hyperplane not intersecting the Roman surface, so we can display front and back 
view (see \Cref{remark:front_back_view}). For the remaining two, we 
choose two hyperplanes at infinity that do not separate special points in order 
to produce front and back view.
\end{example}

\appendix

\section{Computation of conductor ideals}
\label{conductor}

The aim of this appendix is to explain how to compute the image of the map 
\[
 \HHom_{\OO_C} \bigl( \pi_{\ast} \OO_Y, \OO_C \bigr)_c \otimes \hat{\OO}_{C,c} 
\longrightarrow \hat{\OO}_{C,c} \,, 
\]
namely the conductor ideal, when $c \in C$ is a special point of the 
silhouette. We proceed by first determining normal forms for the projection 
around the special points, then computing the conductor ideals in those 
particular situations, and eventually finding equivariant formulas for these 
ideals that can hence be used without reducing the situation to normal forms.

We start by providing normal forms for each of the seven cases of singularities 
of the silhouette. Recall the notation from \Cref{lemma:iso_nonzerodivisors}:
\[
 E = \hat{\OO}_{C, c} 
 \quad \text{and} \quad
 F = \bigoplus_{y_i \colon \pi(y_i) = c} \hat{\OO}_{Y, y_i} \,.
\]
In each case we express the generators of~$F$ as quotients of 
elements in~$E$, as predicted by \Cref{lemma:iso_nonzerodivisors}.
\begin{itemize}[leftmargin=*]
 \item[$\cdot$] \emph{Nodes of the proper silhouette}. It is well-known that 
nodes are $A_1$ singularities, so they are analytically isomorphic to $\{ (x,y) 
\in \C^2 \, \colon \, xy = 0\}$. Since the projection~$\pi$ is an isomorphism 
away from the node, then the preimage of an analytic neighborhood of the 
node~$c$ is constituted of two irreducible smooth curves, each of them 
isomorphic to the two components of $\{ xy = 0 \}$. Hence they are analytically 
equivalent to two disjoint lines, and so we can suppose that
\[
 E = \frac{\C\pp{x,y}}{(xy)}
 \quad \text{and} \quad
 F = \frac{\C\pp{x,y,z}}{\bigl( z(1 - z), xz, y(1-z) \bigr)} \, ,
\] 
and the map $E \longrightarrow F$ is the natural inclusion sending the classes 
of~$x$ and~$y$ in~$E$ to the classes of~$x$ and~$y$ in~$F$. Since $F$ is 
generated, as an $E$-module, by the classes of~$1$ and~$z$, it is enough to 
show that $[z]$ can be expressed as a quotient of two elements $p,q \in E$, 
where $q$ is a non-zerodivisor. We have
\[
 [z] = \frac{[z][x+y]}{[x+y]} = \frac{[y]}{[x+y]}
\]
and $[x+y]$ is not a zerodivisor in~$E$.
 \item[$\cdot$] \emph{Cusps of the proper silhouette}. It is well-known that 
ordinary cusps are $A_2$ singularities, so they are analytically isomorphic to 
$\{ (x,y) \in \C^2 \, \colon \, x^3 - y^2 = 0\}$. The preimage under the 
projection~$\pi$ of an analytic neighborhood of a cusp is a resolution of the 
cusp, so we can suppose 
\[
 E = \frac{\C\pp{x,y}}{(x^3 - y^2)}
 \quad \text{and} \quad
 F = \frac{\C\pp{x,y,z}}{(x - z^2, y - z^3, x^3 - y^2)} \, .
\]
Again, it is enough to express $[z]$ as the quotient of two elements in~$E$, 
and indeed we have $[z] = [y] / [x]$. 
 \item[$\cdot$] \emph{Nodes of the singular image}. This case is similar to the 
one of the node of the proper silhouette, but we have to take into account 
that the fat silhouette has a non-reduced structure. The radical of the 
analytic ideal of node can be hence supposed to be~$(xy)$, so the ideal is of 
the form $(x^a y^b)$. As we saw at the end of the proof of 
\Cref{proposition:generic_iso}, we have $a = b = 2$. So 
\[
 E = \frac{\C\pp{x,y}}{(x^2 y^2)}
 \quad \text{and} \quad
 F = \frac{\C\pp{x,y,z}}{\bigl( z(1 - z), x^2z, y^2(1-z) \bigr)} \, .
\]
We conclude as in the case of the nodes of the proper contour. 
 \item[$\cdot$] \emph{Triple points of the singular image}. A triple point 
of the singular image is the projection of a triple point of the surface. Such 
a point is analytically at the intersection of three smooth manifolds, each of 
which projects isomorphically to the plane. Hence, these manifolds are graphs 
of functions, so they are analytically equivalent to $\{ (x,y,z) \in \C^3 \, 
\colon \, z - f_i(x,y) = 0\}$ for $i \in \{1,2,3\}$ and $f_i$ are analytic 
functions vanishing at~$(0,0)$. By an analytic change of coordinates fixing 
the $(x,y)$-coordinates, we can assume $f_1 = 0$. The projection of the 
singular curve in the plane is the product $f_2 f_3 (f_2 - f_3)$. In the 
plane we have an ordinary triple point, so the tangents at~$(0,0)$ to $\{ f_2 = 
0\}$ and $\{f_3 = 0\}$ are distinct, hence by the inverse function theorem we 
can suppose that $f_2 = x$ and $f_3 = y$. Therefore, we have
\[
 E = \frac{\C\pp{x,y}}{\bigl( x^2y^2(x - y)^2\bigr)}
 \quad \text{and} \quad
 F = \frac{\C\pp{x,y,z}}{\bigl( z(z-x)(z-y), 3z^2 - 2z(x+y) + xy \bigr)} \, ,
\]
where the exponents are justified as in the previous case. 
In this case, $F$ is generated over~$E$ by~$[1]$, $[z]$ and~$[z^2]$. The 
equation $3z^2 - 2z(x+y) + xy = 0$ provides a linear dependence over~$E$ 
between~$[z]$ and~$[z^2]$ that is monic in~$[z^2]$, so it is enough to show 
that $[z]$ can be expressed as a quotients of elements in~$E$. Taking division 
with remainder of~$z(z-x)(z-y)$ by~$3z^2 - 2z(x+y) + xy$ as polynomials 
in~$z$, we get
\[
 [z] = \frac{[x][y][x+y]}{2[x^2] + 2[y^2] - 2[xy]} \, .
\]
 \item[$\cdot$] \emph{Transverse intersections of proper silhouette and 
singular image whose preimages are two distinct points}. Here we have
\[
 E = \frac{\C\pp{x,y}}{( x y^2)}
 \quad \text{and} \quad
 F = \frac{\C\pp{x,y,z}}{\bigl( z(1 - z), xz, y^2(1-z) \bigr)} \,, 
\]
and so $[z] = [y^2] / ([x] + [y^2])$.
 \item[$\cdot$] \emph{Transverse intersections of proper contour and singular 
images whose preimages are pinch points}. We prove that the projection~$\pi$ is 
an isomorphism in this situation, so we have $E = F$. Recall from 
\Cref{proposition:projection_good} that we can assume that the local 
equation of the surface at a pinch point is $x^2y - z^2 = 0$ while keeping the 
projection along the $z$-axis. Its derivative with respect to~$z$ 
is~$2z$, so the fat contour is the plane curve~$x^2y$ inside the plane $z = 0$, 
thus the fat contour projects isomorphically to the fat silhouette. 
\item[$\cdot$] \emph{Tangential intersections of proper silhouette and singular 
image}. Locally, the singular curve is the intersection of two smooth 
components~$S_1$ and~$S_2$ of the surface~$S$, and one of the two, say~$S_1$, 
contains the proper contour. The restriction of the projection to~$S_1$ is a 
$2 \colon 1$ covering branched along a smooth curve; we can choose analytic 
coordinates such that the equation of~$S_1$ is $z^2-y = 0$, the proper contour 
is $y=z=0$, and the proper silhouette is $y=0$. The second component~$S_2$ 
projects isomorphically to the $xy$-plane, hence it has a local analytic 
equation of the form~$z-f$, where $f$ is a function of~$x$ and~$y$. The two 
components of the silhouette are hence $y=0$ and $f^2-y=0$, obtained by 
eliminating~$z$ from the previous equations. We know that the 
intersection multiplicity is~$2$. This implies that the gradient of~$f$ is 
independent from~$y$. Hence we can choose $f=x$ as the third coordinate. In 
this coordinate system, we get
\[
 E = \frac{\C\pp{x,y}}{\bigl( y(x^2-y)^2\bigr)}
 \quad \text{and} \quad
 F = \frac{\C\pp{x,y,z}}{\bigl( (z-x)(z^2-y), z^2-y+2z(z-x) \bigr)} \, .
\]
As in the case of triple points, the module~$F$ is generated by~$[1]$, $[z]$,
and~$[z^2]$. We get quotient representations for these elements in an analogous
way (namely, by polynomial division): 
\[
 [z] = \frac{[4xy]}{[3y+x^2]} \, .
\]
\end{itemize}

\begin{lemma}
\label{lemma:conductor}
 For each of the seven types of special points of the fat silhouette~$C$, the 
conductor ideals of the normal forms provided above are:

\textup{
 \begin{center}
 \begin{tabular}{@{}cc@{}}
  \toprule
  Type of singularity & Conductor ideal \\ \midrule
  Nodes of the proper silhouette & $ \bigl( [x],[y] \bigr)$ \\[1ex]
  Cusps of the proper silhouette & $ \bigl( [x],[y] \bigr)$ \\[1ex]
  Nodes of the singular image & $ \bigl( [x^2],[y^2] \bigr)$ \\[1ex]
  \multirow{2}{*}{Triple points of the singular image} & $\bigl( 
[x^2-xy+y^2],$ \\ & $[xy(x+y)] \bigr)$ \\[1ex]
  Transverse intersections of prop.\ silhouette and sing.\ image  & 
\multirow{2}{*}{$\bigl( [x], [y^2] \bigr)$} \\ whose 
preimages are two distinct points \\[1ex]
  Transverse intersections of prop.\ silhouette and sing.\ image & 
\multirow{2}{*}{$\bigl( [1] \bigr)$} \\ whose preimages are pinch points  
\\[1ex]
  Tangential intersections of prop.\ silhouette and sing.\ image & $\bigl( 
[xy], [3y+x^2] \bigr)$ \\
  \bottomrule
 \end{tabular}
 \end{center}
}
\end{lemma}
\begin{proof}
 We analyze each case separately.
 \begin{itemize}[leftmargin=*]
 \item[$\cdot$] \emph{Nodes of the proper silhouette}. Since $F$ is generated 
over $E$ by $[1]$ and $[z]$, the conductor ideal is $\{ [\alpha] \in E 
\,\colon\, [\alpha z] \in E\}$. Hence we look for $[\alpha] \in E$ such that 
$[\alpha y] = [\beta (x+y)]$ for some $\beta \in \C\pp{x,y}$ (recall the 
description of~$[z]$ as a quotient of elements of~$E$). We calculate (in the 
standard polynomial ring, by means of computer algebra) the intersection of the 
two ideals~$(y, xy)$ and~$(x+y, xy)$, which is~$(xy, y^2)$. This implies that 
the conductor ideal is~$([x],[y])$, because this equals the colon ideal~$(xy, 
y^2) \colon (y)$.
 \item[$\cdot$] \emph{Cusps of the proper silhouette}. As in the previous case, 
it is enough to compute the intersection of the two ideals~$(y, x^3 - y^2)$ 
and~$(x,x^3 - y^2)$, which is~$(y^2, xy, x^3)$. From this it follows that the 
conductor ideal is~$([x],[y])$.
 \item[$\cdot$] \emph{Nodes of the singular image}. This case is analogous to 
the one of nodes of the proper silhouette.
 \item[$\cdot$] \emph{Triple points of the singular image}. This case is 
analogous to the one of nodes of the proper silhouette.
 \item[$\cdot$] \emph{Transverse intersections of proper silhouette and 
singular image whose preimages are two distinct points}. This case is analogous 
to the one of nodes of the proper silhouette.
 \item[$\cdot$] \emph{Transverse intersections of proper contour and singular 
images whose preimages are pinch points}. Since here $E = F$, the conductor is 
the trivial ideal. 
 \item[$\cdot$] \emph{Tangential intersections of proper silhouette and 
singular image}. This case is analogous to the one of nodes of the proper 
silhouette. 
\end{itemize}
\end{proof}

One could think that \Cref{lemma:conductor} provides a way to compute the 
conductor ideals of the special points from the knowledge of the fat 
silhouette: one could think, in fact, of bringing each of the special points to 
the corresponding normal form, and then pick the conductor ideal from the 
table. This would not be correct, since by knowing only the fat silhouette we 
do not have control on the fat contour, and so we cannot ensure that the 
preimages of the special points are in normal form. This seems a hindrance to 
the creation of an algorithm having as input only the fat silhouette, because 
the conductor ideal may depend on the fat contour. We now show that this is not 
the case.

\begin{lemma}
\label{lemma:conductor_invariant}
 The conductor ideals at the special points depend only on the fat silhouette. 
\end{lemma}
\begin{proof}
 We show that the ideals determined in \Cref{lemma:conductor} for the 
normal forms are equivariant under analytic changes of coordinates in the plane, 
thus proving the statement.
 \begin{itemize}[leftmargin=*]
 \item[$\cdot$] \emph{Nodes of the proper silhouette}. In this case, the 
conductor ideal is just the maximal ideal of $\hat{\OO}_{C,c}$.
 \item[$\cdot$] \emph{Cusps of the proper silhouette}. Same situation as for 
the nodes.
 \item[$\cdot$] \emph{Nodes of the singular image}. Here the conductor ideal is 
the sum of the squares of the two ideals defining the two analytic components of the node.
 \item[$\cdot$] \emph{Triple points of the singular image}.
 Let $f$ be an analytic local equation of the fat silhouette at a triple point. 
We then know that we can write $f = h_1 h_2 h_3$ with $h_1 + h_2 + h_3 = 0$ for 
some power series $\{ h_i \}$ of order one. We prove that the conductor ideal 
equals 
\[
 J := \bigl( a_1^2 + a_2^2 + a_3^2 \, \colon \, a_i \in (h_i) \text{ for } i 
\in \{1,2,3\} \text{ and } a_1+a_2+a_3 = 0 \bigr) \,.
\]
Since the latter ideal has a formulation that is equivariant under analytic 
changes of coordinates, it is enough to check that $J$ coincides with the 
conductor ideal in the situation of the normal form, namely when
\[
 h_1 = -x, \quad h_2 = y, \quad h_3 = x - y \,.
\]
Recall that in this case the conductor ideal is $I = (x^2 - xy + y^2, x^2y + 
xy^2)$. We first show the containment $J \subset I$. Consider an element 
in~$J$, namely pick
\[
 a_1 = -\alpha x, \quad a_2 = \beta y, \quad a_3 = \gamma (x-y) = \alpha x - 
\beta y \,.
\]
 for some $\alpha, \beta, \gamma \in \C\pp{x,y}$. This forces $\alpha = \gamma 
 - uy$ and $\beta = \gamma + ux$ for some $u \in \C\pp{x,y}$. A direct 
computation shows that 
\[
 a_1^2 + a_2^2 + a_3^2 = 2 \gamma^2 (x^2 - xy + y^2) + 2\gamma u (x^2y + 
xy^2) + 2u^2 x^2 y^2
\]
and hence $a_1^2 + a_2^2 + a_3^2 \in I$, since one can check that $I$ contains 
$(x,y)^4$. To prove the 
opposite inclusion, it is enough to show that $x^2 - xy + y^2$ and $x^2 y + 
xy^2$ are in~$J$. The first case is immediate, since $2(x^2 - xy + y^2) = 
(-x)^2 + y^2 + (x-y)^2$. For the second element, it is enough to pick the 
two triples $(a_1, a_2, a_3)$ corresponding to $(\gamma,u) = (1,1)$ and to 
$(\gamma,u) = (1,-1)$, and to subtract the corresponding sums of squares.
 \item[$\cdot$] \emph{Transverse intersections of proper silhouette and 
singular image whose preimages are two distinct points}. Here the conductor 
ideal is the sum of the ideal of the proper silhouette and of the square of the 
ideal of the singular image.
 \item[$\cdot$] \emph{Transverse intersections of proper contour and singular 
images whose preimages are pinch points}. In this case the conductor is the 
trivial ideal.
 \item[$\cdot$] \emph{Tangential intersections of proper silhouette and 
singular image}. As we did in the case of triple points of the singular image, we 
provide an equivariant description of the conductor ideal. Consider the situation 
of the normal form, where the conductor ideal is $I = (xy, 3y+x^2)$. Notice 
that it equals the ideal 
\[
 J :=
 \bigl\{ 
  a \in \C \pp{x,y} \, \colon \, a(0,0) = 0 
  \text{ and }
  \operatorname{mult}_{(0,0)} (a, 3y + x^2) \geq 3 
 \bigr\} \,. 
\]
We show that the latter description is equivariant under changes of analytic 
coordinates. Consider the following setting (see 
\Cref{figure:tangential_invariant}): pick an analytic neighborhood of a 
tangential intersection of proper silhouette and singular image, and apply to 
it an analytic isomorphism. Blow up the two analytic neighborhoods at the 
tangential intersection; the previous analytic isomorphism then extends to an 
isomorphism of two neighborhoods of the exceptional divisors, which restricts 
to an automorphism of~$\p^1$ on the exceptional divisors. After the blow up, 
the strict transforms of proper silhouette and singular image intersect 
transversally, and the exceptional divisor passes through that point of 
intersection. A further blow up separates these three curves and introduces a 
second exceptional divisor intersecting each of them transversally.
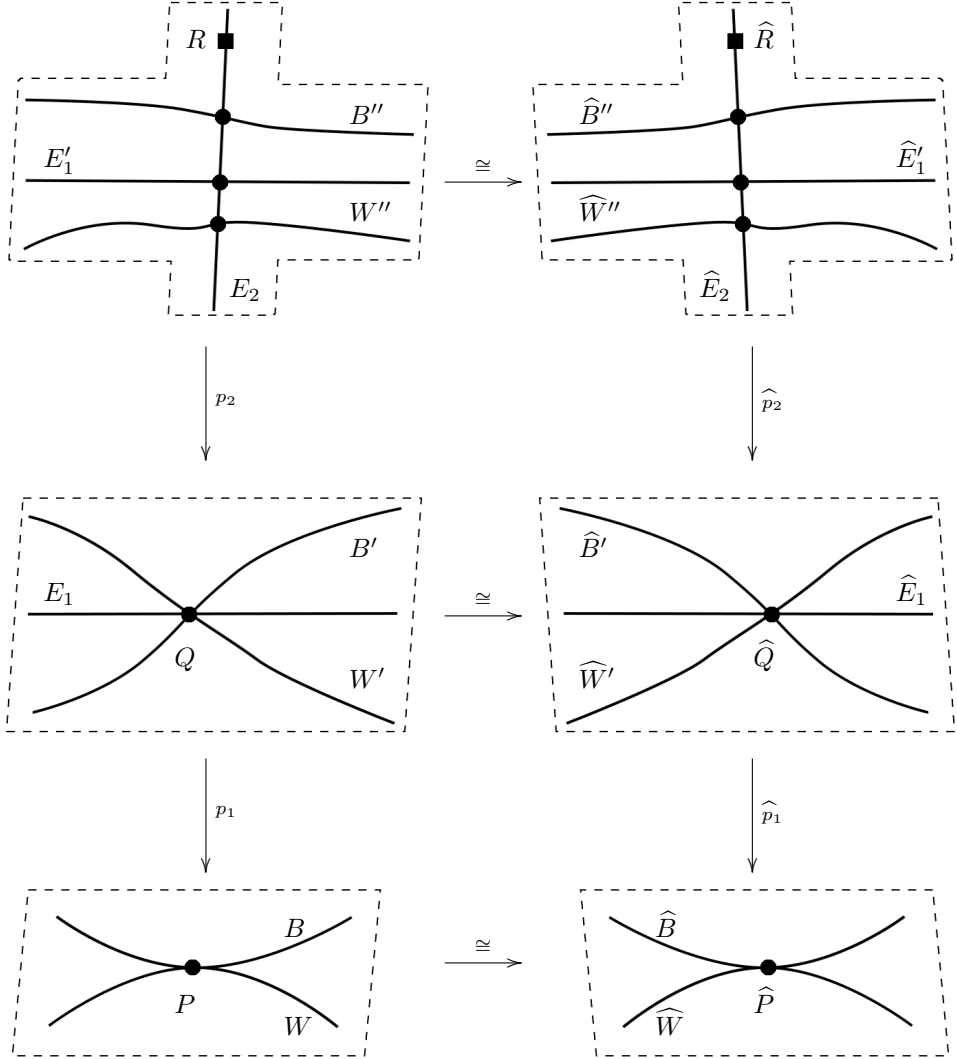
\begin{figure}
 \begin{overpic}[width=\textwidth]{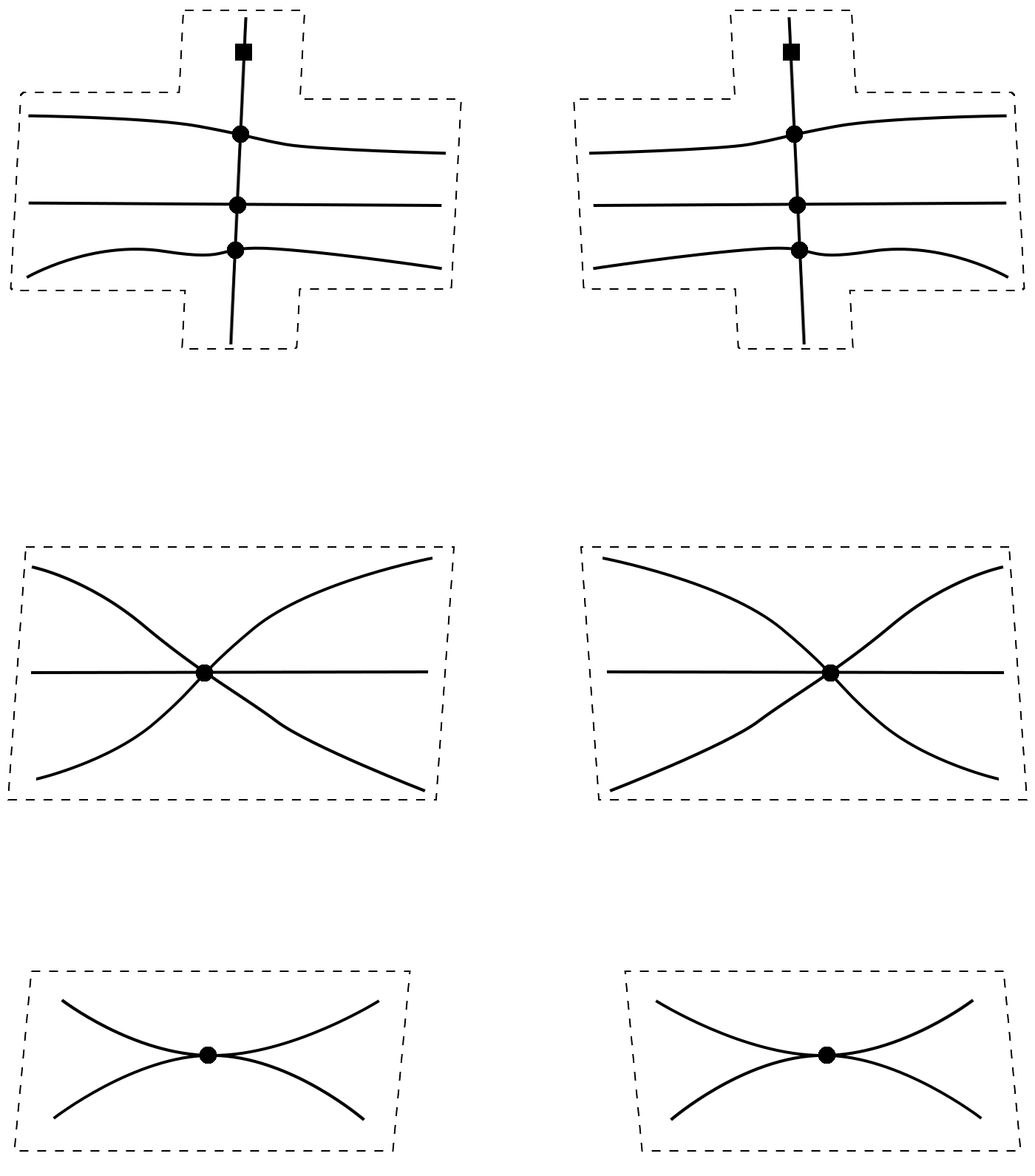}
  \put(39,10){$\xymatrix@C=1cm{ \ar[r]^{\cong} &}$}
  \put(39,42){$\xymatrix@C=1cm{ \ar[r]^{\cong} &}$}
  \put(39,82){$\xymatrix@C=1cm{ \ar[r]^{\cong} &}$}
  \put(18,30){$\xymatrix@R=1.5cm{\mbox{} \ar[d]^{p_1} \\ \mbox{}}$}
  \put(18,68){$\xymatrix@R=1.5cm{\mbox{} \ar[d]^{p_2} \\ \mbox{}}$}
  \put(68,30){$\xymatrix@R=1.5cm{\mbox{} \ar[d]^{\hat{p}_1} \\ \mbox{}}$}
  \put(68,68){$\xymatrix@R=1.5cm{\mbox{} \ar[d]^{\hat{p}_2} \\ \mbox{}}$}
  \put(16,6){$P$}
  \put(16,38){$Q$}
  \put(17,95){$R$}
  \put(4,44){$E_1$}
  \put(4,84){$E_1'$}
  \put(21,72){$E_2$}
  \put(26,13){$B$}
  \put(26,4){$W$}
  \put(32,48){$B'$}
  \put(32,36){$W'$}
  \put(32,88){$B''$}
  \put(32,79){$W''$}
  \put(69,6){$\hat{P}$}
  \put(69,38){$\hat{Q}$}
  \put(69,95){$\hat{R}$}
  \put(82,44){$\hat{E}_1$}
  \put(82,84){$\hat{E}'_1$}
  \put(64,72){$\hat{E}_2$}
  \put(60,13){$\hat{B}$}
  \put(60,4){$\hat{W}$}
  \put(53,48){$\hat{B}'$}
  \put(53,36){$\hat{W}'$}
  \put(53,88){$\hat{B}''$}
  \put(53,79){$\hat{W}''$}
 \end{overpic}
 \caption{Schematic view of the construction used to prove independence of the 
conductor ideal from the fat contour in the case of tangent intersections: we 
blow up the two curves~$B$ and~$W$ at their intersection~$P$, obtaining an 
exceptional divisor~$E_1$, which intersects the strict transforms~$B'$ and~$W'$ 
in a point~$Q$. Blowing up again we introduce another exceptional 
divisor~$E_2$. The conductor ideal can be interpreted as the ideal 
of functions vanishing at~$P$, whose controlled transforms vanish at~$Q$ and at 
a point~$R$ of~$E_2$. If we have an analytic isomorphism around the tangent 
intersection, the fact that it extends to an isomorphism on exceptional 
divisors proves that the image of~$R$ is prescribed, since the isomorphism $E_2 
\longrightarrow \hat{E}_2$ is an automorphism of~$\p^1$ which must preserve the 
intersections of~$E_2$ with the strict transforms~$E_1'$, $B''$ and $W''$.}
 \label{figure:tangential_invariant}
\end{figure}
Let us use the following notation: we denote by~$p_1$ the first blow up map, 
by~$E_1$ its exceptional divisor, by~$B'$ and~$W'$ the strict transforms of~$B$ 
and~$W$; we denote by~$p_2$ the second blow up map, by~$E_2$ its exceptional 
divisor, and by~$E_1'$, $B''$ and~$W''$ the strict transforms 
of~$E_1$, $B'$ and~$W''$. Let $P$ be the tangential intersection of~$B$ 
and~$W$, let $Q$ be the intersection of~$B'$, $W'$ and $E_1$. We show 
that the ideal~$J$ coincides with 
\[
 K :=
 \bigl\{ 
  a \in \C \pp{x,y} \, \colon \, 
  a(P) = 0, \, 
  p_1^{!}(a)(Q) = 0, \, 
  (p_2 \circ p_1)^{!}(a)(R) = 0 
 \bigr\} \,,
\]
where $p_1^{!}(a)$ is the \emph{controlled transform} of~$a$ under the blow up 
map~$p_1$, and $R$ is a point on~$E_2$. To show this, notice that $I$ is 
analytically equivalent to the ideal~$(x^3, y)$, and so $J$ can be described as 
the ideal of functions~$a$ that vanish at~$(0,0)$ and such that 
$\operatorname{mult}_{(0,0)} (a, y) \geq 3$, namely that are of the form $a = 
\sum_{i,j} a_{ij} x^i y^j$ with $i \geq 3$ or $j \geq 1$. Let us compute~$K$ in 
this new setting: the condition $a(P) = 0$ implies that if $a = 
\sum_{i,j} a_{ij} x^i y^j$, then $i \geq 1$ or $j \geq 1$. The controlled 
transform $p_1^{!}(a)$ equals, in the chart with coordinates $(x,\tilde{y})$ 
with $y = \tilde{y}x$, the function $\sum_{i,j} a_{ij} x^{i+j-1} \tilde{y}^j$ 
--- we 
subtract $1$ in the exponent of $x$ since in these coordinates $p_1^{!}(a) = 
p_1^{\ast}(a) / x$ because $x = 0$ is the equation of the exceptional divisor; 
so, the condition $p_1^{!}(a)(Q) = 0$ translates into $i + j - 1 \geq 1$ or $j 
\geq 1$. Similarly, the condition $(p_2 \circ p_1)^{!}(a)(R) = 0$ translates to 
$i + 2j - 2 \geq 1$ or $j \geq 1$. One can check that the conjunction of these 
three conditions is equivalent to the condition defining~$J$. Hence the 
conductor ideal~$I$ equals~$K$, and we see that the description of~$K$ is 
equivariant under local analytic changes of coordinates, because, as already 
mentioned, any such change extends to an isomorphism at the level of the 
exceptional divisors~$E_2$, which are projective lines; since this isomorphism 
must preserve the intersections of~$E_2$ with $E_1'$, $B''$ and~$W''$, it is 
uniquely determined, and so the same holds for the image of~$R$ under it. This 
proves that the formation of the conductor ideal is equivariant.

\end{itemize}
\end{proof}

We conclude this appendix providing formulas to compute the conductor ideals of 
special points without the need of bringing the equation of the fat silhouette 
to a normal form. The proof of \Cref{lemma:conductor_invariant} clarifies 
how to do so in the case of nodes and cusps of the proper silhouette, and of 
transversal intersections of proper silhouette and singular image (coming both 
from pinch points or pairs of distinct points). We are hence 
left with:
\begin{description}
 \item[Nodes of the singular image]
 If $f$ is a local analytic equation of the silhouette, namely of the 
reduced structure of the fat silhouette, one sees that the conductor ideal in 
the normal form is generated by the $2 \times 2$ minors of the matrix
\[
\begin{pmatrix}
 \partial_{xx} f & \partial_{xy} f & \partial_{yy} f & \partial_x f & 
\partial_y f \\
 \partial_{xx} (f^2) & \partial_{xy} (f^2) & \partial_{yy} (f^2) & \partial_x 
(f^2) & \partial_y (f^2) 
\end{pmatrix} \, .
\]
 We show that these formulas are equivariant under analytic changes of 
coordinates, and so they can be used for any node of the singular image. First 
of all, notice that the conductor ideal for the normal form contains the 
ideal~$(x,y)^3$. 
This means that the conductor ideal always contains the third power of the 
maximal ideal of the point. Hence, in order to prove that the formulas we give 
are equivariant, it is enough to consider their part of order at most two. 
Locally analytically, the function~$f$, which is of order two at the node, 
splits as a product $f = h_1 h_2$, where each~$h_i$ has order one. We show that 
any perturbation of the~$h_i$ by an element of order at least two does not 
influence the equivariance property. In fact, suppose that we write $h_1 = 
\tilde{h}_1 + \varepsilon$, where $\tilde{h}_1$ has order one and $\varepsilon$ 
has order two. Then $f = \tilde{h}_1 h_2 + \varepsilon h_2$, and so 
$\varepsilon h_2$ has order three. This means that any of the second 
derivatives of~$f$ will be affected by a perturbation of order one. With 
similar computations, one sees that the second derivatives of~$f^2$ are 
affected by a perturbation of order two. This implies that any minor of the 
previous matrix is affected by a perturbation of order at least three, which 
can be ignored since the conductor ideal contains the whole third power of the 
maximal ideal. 
Hence, in order to prove equivariance, it suffices to check that the formula 
we propose is equivariant under all coordinate changes of the form
\[ 
 \begin{pmatrix}
  x \\ y
 \end{pmatrix}
 \mapsto 
 \begin{pmatrix} 
  a_1 \, x + a_2 \, y \\
  a_3 \, x + a_4 \, y 
 \end{pmatrix} \,,
\]
namely that it always provides the conductor ideal, which is given by~$(h_1^2, 
h_2^2)$.
These checks can performed for symbolic parameters $a_1, \dotsc, a_4$ with 
the help of a computer algebra system. We implemented these tests in a Maple 
script inside the package we developed, see the introduction for the 
Internet address where to find the code.
 \item[Triple points of the singular image]
 If $f$ is a local analytic equation for the silhouette, then one can check 
that the conductor ideal in the normal form is generated by the $2 \times 2$ 
minors of the matrix
\[
\begin{pmatrix}
 \partial_{xxx} f & \partial_{xxy} f & \partial_{xyy} f & \partial_{yyy} f \\
 \partial_{xxx} (f^2) & \partial_{xxy} (f^2) & \partial_{xyy} (f^2) & 
\partial_{yyy}(f^2)
\end{pmatrix}
\]
together with
\begin{align*}
 3\, f \, f_{xxx}\, f_{xyy}\, f_{xyyy} - 
 3\, f \, f_{xxy}^{2} f_{xyyy} - 
 3\, f \, f_{xxx} \, f_{yyy} \, f_{xxyy} + 
 3\, f \, f_{xxy} \, f_{xyy} \, f_{xxyy} + \\
 3\, f \, f_{xxy} \, f_{yyy} \, f_{xxxy} - 
 3\, f \, f_{xyy}^{2} \, f_{xxxy} + 
 2\, f_{xx} \, f_{xy} \, f_{xxy} \, f_{yyy} - 
 2\, f_{xx} \, f_{xy} \, f_{xyy}^{2} + \\
 2\, f_{xx} \, f_{yy} \, f_{xxx} \, f_{yyy} - 
 2\, f_{xx} \, f_{yy} \, f_{xxy} \, f_{xyy} - 
 4\, f_{xy}^{2} \, f_{xxx} \, f_{yyy} + 
 4\, f_{xy}^{2} \, f_{xxy} \, f_{xyy} + \\
 2\, f_{xy} \, f_{yy} \, f_{xxx} \, f_{xyy} - 
 2\, f_{xy} \, f_{yy} \, f_{xxy}^{2} \, .
\end{align*}
The last element has been computed by imposing that a symbolic linear 
combination of a list of candidates is in the conductor ideal for several 
randomized examples. As for the case of nodes, we show that these formulas are 
equivariant under analytic changes of coordinates. Since the conductor ideal of 
the normal form contains the ideal~$(x,y)^4$, the conductor ideal always 
contains the fourth power of the maximal ideal of the point, and so we can 
neglect contributions of order at least four in the formulas. We know that for 
a triple point we always have, locally analytically, the factorization $f = h_1 
h_2 (h_1 - h_2)$, where each~$h_i$ has order one. Similarly as before, a direct 
inspection of the formulas shows that a perturbation of order at least three of 
the~$h_i$ determines a perturbation of order at least four in the formula. 
Thus it is sufficient to check that the formulas are equivariant under changes 
of coordinates of the form
\[ 
 \begin{pmatrix}
  x \\ y
 \end{pmatrix}
 \mapsto 
 \begin{pmatrix} 
  a_1 \, x + a_2 \, y + b_1 \, x^2 + b_2 \, xy + b_3 \, y^2, \\
  a_3 \, x + a_4 \, y + b_4\,  x^2 + b_5 \, xy + b_6 \, y^2
 \end{pmatrix} \,,
\]
namely that they always provide the conductor ideal, which is given by
\[
 \bigl(
  h_1^2 + h_2^2 + (h_1 - h_2)^2, \,
  h_1^2 h_2 + h_1 h_2^2
 \bigr) \,.
\]
This is checked symbolically with the aid of computer algebra.
 \item[Tangential intersections of proper silhouette and singular image]
 In this case, if $f$ and $g$ are local analytic equations of 
the proper silhouette and of the singular image,  the 
conductor ideal in the normal form is given by
\[
 \bigl( 
  fg, \,
  4f \tth \partial_y g - g \tth \partial_y f, \,
  4f \tth \partial_x g - g \tth \partial_x f
 \bigr) \, . 
\]
The proof of equivariance follows as in the previous cases.
\end{description}

\section*{Acknowledgments}
We thank Kristian Ranestad for pointing out to us the work of Chisini in 
occasion of the workshop ``Meeting on Algebraic Vision'' organized at TU Berlin 
on October 8-9, 2015. Matteo Gallet thanks Emilia Mezzetti and Dario Portelli 
for providing several useful suggestions and references to the existing 
literature about the questions investigated in this paper. Some of the pictures 
of the surfaces have been realized by the free software 
\texttt{surfex}~\cite{surfex}, and others by the free software 
\texttt{POV-Ray}~\cite{povray}.

\end{document}